\documentclass[10pt,a4paper]{article}

\usepackage[leqno]{amsmath}
\usepackage{amssymb,amsthm,upref,amscd}
\usepackage[T1]{fontenc}
\usepackage{times}
\usepackage{cite}
\usepackage{amsfonts}
\usepackage[colorinlistoftodos]{todonotes}
\usepackage{setspace}

\usepackage{color}
\usepackage{hyperref}
\usepackage{graphicx}
\usepackage{wrapfig}
\usepackage{setspace}

\usepackage[usenames,dvipsnames]{pstricks}
\usepackage{epsfig}

\newtheorem{thm}{Theorem}[section]
\newtheorem{lem}[thm]{Lemma}

\newtheorem*{ack}{Acknowledgments}

\theoremstyle{definition}
\newtheorem{definition}[thm]{Definition}

\theoremstyle{remark}
\newtheorem*{notation}{Notation}
\theoremstyle{remark}
\newtheorem{remark}[thm]{Remark}
\numberwithin{equation}{section}

\newcommand{\C}{{\mathbb C}}
\newcommand{\N}{{\mathbb N}
}

\newcommand{\R}{{\mathbb R}}

\definecolor{blu}{rgb}{0,0,1}

\newcommand{\beq[1]}{\begin{equation}\label{eq:#1}}
\newcommand{\eeq}{\end{equation}}

\begin{document}
	
\title{Stationary waves with prescribed $L^2$-norm for the planar Schr\"odinger-Poisson system}

\author{\sc{Silvia Cingolani} 
	\and
	\sc{Louis Jeanjean}}

\date{}
\maketitle

\begin{abstract}
	\noindent
The paper deals with the existence of standing wave solutions for the Schr\"odinger-Poisson system with prescribed mass in dimension $N=2$.  This leads to investigate the existence of normalized solutions for an   integro-differential equation involving a logarithmic convolution potential, namely 
\begin{equation*}
\left \{
\begin{aligned}
- \Delta u & +  \gamma \Bigl(\log {| \cdot |} * |u|^2 \Bigr) u =a |u|^{p-2} u
\qquad \text{in $\R^2$,} \\ 
&\int_{\R^2}  |u|^2 dx = c
\end{aligned}
\right.
\end{equation*}
where $c>0$ is a given real number. 
Under different assumptions on  $\gamma \in \R$, $a \in \R$, $p>2$, we prove several existence and multiplicity results. 
With respect to the related higher dimensional cases, the presence of the logarithmic kernel, which is unbounded from above and below, makes the structure of the solution set much richer and it forces the implementation of new ideas to catch the normalized solutions. 
\end{abstract}


\section{Introduction}\label{Introduction}

We consider the Schr\"odinger-Poisson system  of the type
\begin{equation}
\left \{
\begin{aligned}
	i \psi_t  &- \Delta \psi +  \gamma w \psi =  a |\psi|^{p-2} \psi ,\\
	&\Delta w = |\psi|^2
\end{aligned}
\right.
\quad \text{in $\R^N \times \R$}
\label{prob}
\end{equation}
where $\psi : \R^N \times \R \to \C$  is the (time-dependent) wave
function, $\gamma \in \R$, $a \in \R$, $p>2$. The function $w$ represents an internal potential for a nonlocal self-interaction of
the wave function $\psi$. The standing wave ansatz $\psi(x,t)= e^{-i
\lambda t} u(x)$, $\lambda \in \R$, reduces (\ref{prob}) to the system
\begin{equation}
\label{prob1}
\left\{
\begin{aligned}
- \Delta u &+ \lambda u + \gamma \, w u = a |u|^{p-2} u,\\
&\Delta w = u^2
\end{aligned}
\right.
\qquad \text{for $u: \R^N \to \R$}.
\end{equation}
The second equation determines $w: \R^N \to \R$ only up to harmonic
functions, but it is natural to choose $w$ as the Newton potential of
$u^2$, i.e., the convolution of
$u^2$ with the fundamental solution $\Phi_N$ of the Laplacian. With
this formal inversion of the second equation in (\ref{prob1}), we
obtain the integro-differential equation
\begin{equation}
\label{eq:56}
- \Delta u + \lambda  u +  \gamma [\Phi_N * |u|^2] u =a |u|^{p-2} u
\qquad \text{in $\R^N$,}
\end{equation}
where $\Phi_N(x)= -\frac{1}{N(N-2)\omega_d}|x|^{2-N}$ in case $N \ge 3$
and $\Phi_N(x)=\frac{1}{2\pi} \log |x|$ in case $N=2$. Here, as usual,
$\omega_N$ denotes the volume of the unit ball in $\R^N$.

Due to its physical relevance in physics, the system has been extensively studied and it is quite well understood in the case $N \geq 3$.  In particular variational methods are employed to derive existence and multiplicity results of standing waves solutions 
 \cite{AmRu,AzPo,Lieb,penrose,CiCaSe,morozschaftingen,pekar} and  \cite{Lions,BeJeLu,JeLu,luo} for standing wave solutions with prescribed $L^2$-norm.

 In two dimensions, due to the logarithmic nature of its convolution
 kernel, the nonlocal nonlinearity exhibits some serious mathematical differences to
 the higher dimensional case.
 The study of planar nonlocal problems (\ref{eq:56}) remained for a long time an open field of investigation, apart from some 
 numerical studies suggesting the existence of bound states \cite{harrison}.
 
In contrast with the higher-dimensional case $N \geq 3$, the applicability of variational methods is not straightforward for \(N = 2\). 
Although (\ref{eq:56}) has, at least formally, a variational structure related
to the energy functional
$$
I(u)= \frac{1}{2} \int_{\R^2} \bigl(|\nabla u|^2 
+  \lambda u^2 \bigr) dx +\frac{\gamma}{8 \pi} \int_{\R^2}\int_{\R^2}
\log(|x-y|^2) {|u (x)|}^2 {|u (y)|}^2  dx dy - 
\frac{a}{p} \int_{\R^2} | u|^p  dx  
$$
this energy functional is not well-defined on the natural Sobolev space $H^1(\R^2)$.

Inspired by \cite{Stubbe}, T. Weth and the first author \cite{CiWe} developed a variational framework to deal with the equation (\ref{eq:56}),  within the smaller Hilbert space
$$
X:= \Bigl\{u \in H^1(\R^2)\ | \ \int_{\R^2} \log(1+|x|) |{u (x)}|^2 dx < \infty\Bigr\},
$$
endowed with a norm defined for each function \(u \in X\) by
\[
\|u \|_X^2:= \int_{\R^2} {|\nabla u(x)|}^2 + {|u (x)|}^2 \bigl(1 + \log(1+ |x|) \bigr) dx.
\] 
Even if $X$ provides a variational framework for $(\ref{eq:56})$, 
some difficulties however arise in the application of variational arguments, since   the norm of $X$ 
is not invariant under translations whereas the functional $I$ is invariant under translations of \(\mathbb{R}^2\) and  the quadratic part of the functional $I$  is never coercive on $X$. In \cite{CiWe}, for $\lambda >0$ fixed, the authors constructed a sequence of solution pairs $(\pm u_n)_{n \in \mathbb{N}} \subset X$  to the equation \eqref{eq:56}  such that $I(u_n) \to \infty$ as $n \to + \infty$, under the assumption $p \geq 4$ and $\gamma >0$, $a \geq 0$. They also provided a variational characterization of the least energy solution.
Successively,  Du and Weth proved the existence of ground state solutions and of infinitely many nontrivial changing sign  solutions for $(\ref{eq:56})$ when $2<p<4$.
When $a=0$, $\gamma >0$, the equation $(\ref{eq:56})$ is also referred to as the planar Choquard equation and it can be derived
from the Schr\"odinger-Newton \cite{penrose}. In \cite{CiWe}, it has been showed that every positive solution $u \in X$ of \eqref{eq:56} is radially symmetric up to translation and strictly decreasing in the distance
from the symmetry center.  Moreover $u$ is unique up to translation in \(\R^2\).
In \cite{BoCiVa}, Bonheure, Van Schaftingen and the first author  obtained sharp decay estimates of this unique positive solution to the logarithmic Choquard equation \eqref{eq:56} and they showed the nondegeneracy of the unique positive ground state.  We also mention the recent paper \cite{BatVanS} for the existence of the ground state of \eqref{eq:56},  with $a=0$, $\gamma =1$, via relaxed problems.

In the present paper we are interested to study existence of standing waves solutions for the planar Schr\"odinger-Poisson system with prescribed mass, which is a physically relevant  open problem.
To this aim, for any $c \in \R$, $c > 0$, we consider the problem of finding of 
solutions to
\begin{equation}
\left \{
\begin{aligned}
- \Delta u & +  \gamma \Bigl(\log {| \cdot |} * |u|^2 \Bigr) u =a |u|^{p-2} u
\qquad \text{in $\R^2$,} \\ 
&\int_{\R^2}  |u|^2 dx = c.
\end{aligned}
\right.
\label{eq:59}
\end{equation}
Solutions to (\ref{eq:59}) can be obtained as critical points of the energy functional  
$$F(u)=\frac{1}{2}A(u)+\frac{\gamma}{4}V(u)-\frac{a}{p}\; C(u)$$
where  
$$A(u)=\int_{\R^ 2}|\nabla u(x)|^ 2\;dx, \quad V(u)=\iint_{\R^ 2\times \R^ 2}|u(x)|^2|u(y)|^2\; \log(|x-y|)\,dx\,dy$$
and 
$$ C(u)=\int_{\R^2}|u(x)|^p\;dx$$ under the constraint
\begin{equation}
S(c)=\{u\in X \; | \; \|u\|_{L^2}^2 =c\}.
\end{equation}
If $p>2$, $F$ is well defined and $C^1$ on $X$ (see for example \cite[Lemma 2.2]{CiWe})
and any critical point $u$ of $F_{|S(c)}$ corresponds to a solution of (\ref{eq:56}) where the parameter $\lambda \in \R$ appears as a Lagrange multiplier.

We shall seek for normalized solutions to $(\ref{eq:59})$ using variational arguments and 
we address a situation which is substantially different compared to those considered in the three dimensional case
\cite{Lions,BeJeLu,JeLu,luo}, since the logarithmic kernel changes it sign and  the energy functional can  unbounded from above and below on the constraint. This forces the implementation of new ideas to catch the normalized solutions.

As a first main result, we explicit conditions under which the functional $F$ is bounded from below on $S(c)$ and the infimum
\begin{equation}\label{definfimum}
m= \underset{S(c)}{\inf} F
\end{equation}
 is achieved. We prove the following result.

\begin{thm}\label{minimization1}
Assume $\gamma >0$ and that one of the three following conditions holds:
$$ \mbox{(i) } a \leq 0 \mbox{ and } p>2, \quad \mbox{(ii) }  a > 0  \mbox{ and } p<4, \quad \mbox{(iii) }  a > 0, p=4  \mbox{ and } c < \displaystyle \frac{2}{a K_{GN}},$$
where ${K_{GN}}$ is the best constant of the Gagliardo-Nirenberg inequality (\ref{GN}).
Then the infimum $m$ defined in (\ref{definfimum}) is achieved. In addition any minimizing sequence has, up to translation, a subsequence converging strongly in $X$.
\end{thm}

Note that the property of convergence of the minimizing sequences insured by Theorem \ref{minimization1} provides a strong indication that the set of standing waves associated to the set of minimizers for $F$ on $S(c)$ is orbitally stable. \smallskip

In all the other cases that we shall now consider the functional $F$ will be unbounded from below on $S(c)$ and, in particular, it will not be possible to find a global minimizer. To overcome  this difficulty we shall exploit the property that $F$, restricted
to $S(c)$, possesses a natural constraint, namely a set, that we denote by $\Lambda(c)$, that contains all the critical points of $F$ restricted to $S(c)$.

Precisely, for each $u\in L^2(\R^2)$ and $t>0$, we consider the dilations
$$u^t(x): =t\;u(tx)  \quad \text{ for all } x\in \R^2,$$ 
which define an action of the group $((0,\infty),\times)$ on $S(c)$, since $||u^t||_{L^2}^2= ||u||_{L^2}^2$.
By easy computations, we also get
\begin{equation}\label{scalings}
A(u^t)=t^2\; A(u), \quad C(u^t)=t^{p-2}\;C(u) \quad \mbox{and} \quad
V(u^t)=V(u)-c^2\; \log t.
\end{equation}
Defining the {\sl fiber map} $t \in (0, \infty) \mapsto g_u(t):= F(u^t)$,  we can derive the formula
\begin{equation}\label{LinkF-Q}
\frac{d}{dt} F(u^t)=\frac{Q(u^t)}{t},
\end{equation}
where we have set
\begin{equation}\label{defQ}
Q(u):=\left. \frac{d}{dt}\right|_{t=1}F(u^t)=A(u)-a\;\frac{p-2}{p}\; C(u)-\gamma\;\frac{c^ 2}{4}.
\end{equation}

Actually the  condition $Q(u)=0$ corresponds to a Pohozaev identity and  the set 
$$
\Lambda(c) := \{u \in S(c) \  | \ Q(u) =0\} = \{u \in S(c) \ | \ g_u'(1)=0\}
$$
appears as a natural constraint. As we shall see, in Lemma \ref{coercive}, when $\gamma >0$, $F$ restricted to $\Lambda(c)$ is bounded from below.

We also recognize that for any $u \in S(c)$, the dilated function $\textcolor{blue}{u^s(x)= s u (s x)}$ 
belongs to the constraint $\Lambda(c)$ if and only if $s \in \R$ is a critical value of the fiber map $t \in (0, \infty) \mapsto g_u(t)$, namely $g'_{u} (s)=0$.  Moreover  it happens that  $g'_{u} (s) =
\textcolor{blue}{g'_{u^s}(1)}$, so that
 if $s$ is a critical point of $g_u$, then 
 $\textcolor{blue}{u^s}$ can be seen {\sl as a projection} of $u$ on the set $\Lambda(c)$.

Now setting
$$c_0=2\left[\frac{p(p-4)^{\frac{p-4}{2}}}{(p-2)^\frac{p}{2}}\frac{1}{a\gamma^{\frac{p-4}{2}}K_{GN}}\right]^ {\frac{1}{p-3}},$$ 
we show that if  $\gamma>0$, $a>0$, $p>4$, and $c <c_0$, then the  set $\Lambda(c)$ is a submanifold of $X$ of codimension $2$ and a submanifold of $S(c)$ of codimension $1$.

At this stage, is view of the  geometric profile of $g_u(s)$, and inspired by \cite{Ta}, see also \cite{So1} for a very recent applications of this idea, we are lead to decompose $\Lambda(c)$ into three disjoint subsets 
$$
\Lambda^+(c) = \{ u \in S(c) \ |\ g'_u(1)=0, \ g''_u(1) > 0 \}
$$
$$
\Lambda^-(c) = \{ u \in S(c) \ |\ g'_u(1)=0, \ g''_u(1) < 0 \}
$$
$$
\Lambda^0(c) = \{ u \in S(c) \ |\ g'_u(1)=0, \ g''_u(1) = 0 \}.
$$
Firstly we recognize that for any $u \in S(c)$, there exist an unique $s^+_u >0$ such that $u^{s^+} \in \Lambda^+(c)$ 
and an unique $s^-_u >0$ such that $u^{s^-} \in \Lambda^-(c)$. Such $s^+_u$ and $s^-_u$ are respectively strict local minimum point and strict local maximum point for $g_u$. Finally setting
$$
I^+: S(c) \to \R \quad I^+(u) = F(s^+_u u(s^+_u u(x)))
$$
and
$$
I^-: S(c) \to \R \quad I^-(u) = F(s^-_u u(s^-_u u(x))),
$$
we pass to minimize the functionals $I^\pm$ on $S(c)$, which correspond to minimize $F$ on $\Lambda^\pm(c)$.

Precisely, setting
$$\gamma^+(c) := \inf_{\Lambda^+(c)} F(u) \quad \mbox{and} \quad \gamma^-(c) := \inf_{\Lambda^-(c)} F(u),$$
we prove the following result.

\begin{thm}\label{infreachxxxx}
Let  $\gamma>0$, $a>0$, $p>4$, $c < c_0$. Then $\Lambda^0(c) = \emptyset$, while $\Lambda^\pm(c)$ are not empty and 
there exist	
	$$u^+ \in \Lambda^+(c)  \ \ \mbox{such that} \ \ F(u^+)= \gamma^+(c) \quad \mbox{and} \quad u^- \in \Lambda^-(c) \ \ \mbox{such that} \ \ F(u^-)= \gamma^-(c).$$
In addition $u^+$ and $u^-$ are critical points of $F$ restricted to $S(c)$.
\end{thm}

We remark that the first solution $u^+$, which appears in Theorem \ref{infreachxxxx} as a global  minimizer of $F$ restricted to  $\Lambda(c)$, can also be characterized as a local minimizer of $F$ on the set $S(c) \cap A_{k_0}$ where
$$A_{k_0}= \{ u \in X \,| \, A(u) \leq k_0\} \quad \mbox{where} \quad k_0 = \frac{(p-2)}{(p-4)}\frac{\gamma c^2}{4},$$
see Theorem \ref{minimization local}. Also the second solution $u^-$ corresponds to a critical point of mountain-pass type for $F$ on $S(c).$ The existence of two critical points on $S(c)$, one being a local minimizer and the second one of mountain-pass type is reminiscent of recent works \cite{BJ2, GoJe, JeLuWa, So1} where a similar structure have been observed for prescribed norm problems. \smallskip

Regarding the existence of more than two solutions we derive the following result.

\begin{thm} \label{double infinity}
Let $\gamma >0$, $a>0$, $p>4$ and $c < c_0$. Then $F$ constrained to $S(c)$ possess an infinity of critical points lying on $\Lambda^+(c)$ and an infinity of critical points lying on $\Lambda^-(c)$. These critical points correspond to radially symmetric functions.
\end{thm}

Next we consider the case  $\gamma <0$ which appears more involved than the case $\gamma >0$. Note in particular that when $\gamma <0$ and for $\lambda >0$ fixed, there are still no results of existence or non-existence of solutions to (\ref{eq:56}) set on $\R^2$. \smallskip

Firstly, we notice that if $a \leq 0$ and $p>2$,  for each $u \in S(c)$ the  {\sl fiber map}
$g_u(t):= F(u^t)$ is strictly increasing and so we can state the following non-existence result.

\begin{thm} \label{nonexistence}
	Let $\gamma <0$, $a\leq 0$ and $p>2$. Then $F$ do not has critical points on $S(c)$.
\end{thm}

\smallskip
Concentrating now on the case $\gamma <0$, $a >0$ and $p<4$ we observe that for  $K_1 \in \R$ given by
$$K_1=\frac{1}{2^\frac{4-p}{2}}\frac{1}{K_{GN}}\;\frac{p}{2^{3-p}(p-2)^{\frac{p}{2}}(4-p)^\frac{4-p}{2}}$$
we have, see Lemma \ref{fuori},
$$\Lambda(c) \neq \emptyset \quad \quad \mbox{if and only if} \quad  \quad a \geq K_1\; \gamma^\frac{4-p}{2}\; c^{3-p},$$
and, see Lemma \ref{unboundedbelow},
$$\inf_{\Lambda(c)} F(u) = - \infty, \quad \mbox{if} \quad a> K_1\; \gamma^\frac{4-p}{2}\; c^{3-p}.$$
However, see Lemma \ref{coercive2},
$$
\sup_{\Lambda(c)} F(u) <\infty, \quad \mbox{if} \quad a \geq K_1\; \gamma^\frac{4-p}{2}\; c^{3-p}
$$
and, setting $K_2=2^\frac{4-p}{2}\; K_1$, we are able to show that 
$\Lambda(c)$ is a submanifold, of class $C^1$, of codimension $2$ of $X$ and a submanifold of codimension 
$1$ in $S(c)$ if 
$K_1\; \gamma^\frac{4-p}{2}\; c^{3-p}\leq a< K_2\; \gamma^\frac{4-p}{2}\; c^{3-p}$ and then that 
$\sup_{\Lambda(c)} F(u) $ is achieved by a critical point of $F$ to $S(c)$ (see Theorem \ref{lem_T(c)sousvariete234}).

In the aim to find more than one solution, we may now try to follow the approach, relying on the decomposition of the natural constraint 
 $\Lambda(c)$  into three disjoint subsets 
$\Lambda^+(c)$, $\Lambda^0(c)$ and  $\Lambda^-(c)$, developed when $\gamma >0$. At this point we face a new difficulty. For any choice of $a$ and $c$ there always exists a $u \in S(c)$ such that $u^t \not \in \Lambda(c)$ for any $t>0$. Namely an arbitrary $u \in S(c)$ cannot always be projected on $\Lambda(c)$.\smallskip

To overcome this problem our idea is, roughly speaking,  to introduce an open subset $V$ of $S(c)$, such that for any $u \in V$  the dilation  $u^t \in V$ for any $t>0$ and there exists an unique $s^-_u >0$  such that $u^{s^-_u} \in \Lambda^-(c)$ and an unique $s^+_u >0$  such that $u^{s^+_u} \in \Lambda^+(c)$. 
Such values $u^{s^-_u}$ and $u^{s^+}$ are respectively strict local maximum and strict local minimum point of $g_u$. This geometry holds as soon as $a > K_1 \gamma^{\frac{4-p}{2}} c^{3-p}$ and it makes sense to define the functionals
$$
I^+: V \to \R \quad I^+(u) = F(s^+_u u(s^+_u u(x))),
$$
$$
I^-: V \to \R \quad I^-(u) = F(s^-_u u(s^-_u u(x)))
$$
and to try to maximize the functionals $I^\pm$ on $V$, which correspond to maximize $F$ on $\Lambda^\pm(c) \cap V$.

However, since $V$ has a boundary, we need to insure that our deformation arguments take place inside $V$. The additional condition $ a < K_2 \gamma^{\frac{4-p}{2}} c^{3-p}$ insures that $\Lambda(c) \subset V$ and that the superlevels of $I^\pm$ are complete. Actually we show that if $v_n \to v_0 \in \partial V$ strongly in $X$, then $I^\pm(v_n) \to - \infty$, see Lemma \ref{boundary}. At this point setting
$$\gamma^{+}(c) := \sup_{\Lambda^{+}(c)} F(u) \quad \mbox{and} \quad \gamma^-(c) := \sup_{\Lambda^-(c)} F(u)$$
we are able to prove the following result.
\begin{thm}\label{infreach233}
Assume that $\gamma <0$ and $p<4$. For
$ \displaystyle K_1\; \gamma^\frac{4-p}{2}\; c^{3-p} < a < K_2\; \gamma^\frac{4-p}{2}\; c^{3-p} $
there exist	
$$u^- \in  \Lambda^-(c)  \ \ \mbox{such that} \ \ F(u^-)= \gamma^-(c) \ \ \mbox{and} \ \ u^+ \in \Lambda^+(c) \ \ \mbox{such that} \ \ F(u^+)= \gamma^+(c).$$
In addition $u^-$ and $u^+$ are critical points of $F$ restricted to $S(c)$.
\end{thm}

We remark that the case $\gamma <0$, $a>0$ and $p>4$ seems completely open. Under these assumptions, the geometric picture is somehow simpler than when $p <4$, in particular for any $u \in S(c)$ there exists a unique $t>0$ such that $u^t \in \Lambda(c)$ but what is unclear is how to identify a possible minimax level. 

We end this introduction by mentioning that in the case $\gamma <0$ the existence of more than two solutions remains an open, challenging  problem.
\medskip

The paper is organized as follows. In Section \ref{Preliminary results}, we establish some preliminaries. Section \ref{gammapositif} is devoted the case $\gamma >0$. In Subsection \ref{subsection31} we give the proof of Theorem \ref{minimization1} and in Subsection \ref{subsection32} the one of Theorem \ref{minimization local}. Subsections \ref{subsection33} and \ref{subsection34} are  devoted to the proofs of Theorems \ref{infreachxxxx} and \ref{double infinity} respectively. Section \ref{gammanegatif} deals with the case where $\gamma <0$. In Subsection \ref{propertiesLambda} we derive some properties of $\Lambda(c)$. In Subsection \ref{Onesolution} we give the proof of Theorem \ref{lem_T(c)sousvariete234} and in Subsection \ref{Twosolution} the one of Theorem \ref{infreach233}.

\begin{ack}
{\rm S. Cingolani is  member of the Gruppo Nazionale per l'Analisi Matematica, la Probabilit\`{a} e le loro
	Applicazioni (GNAMPA) of the Istituto Nazionale di Alta Matematica (INdAM).
	 This work has been carried out in the framework of the Project NONLOCAL (ANR-14-CE25-0013), funded by the French National Research Agency.}
\end{ack}

\begin{notation}
In this paper we denote for any $1\leq p<\infty,$ by $L^p(\R^2)$  the usual Lebesgue space with norm
$\|u\|_p^p := \int_{\R^2}|u|^p\,dx,$
and by $H$ the usual Sobolev space $H^1(\R^2)$ endowed with the norm
$\|u\|^2 := \int_{\R^2}|\nabla u|^2+|u|^2 \, dx.$
We denote by $\rightarrow$ and $\rightharpoonup$ the strong convergence and the weak convergence, respectively. We shall write $\varlimsup$ for $\limsup$ and $\varliminf$ for $\liminf$.
\end{notation}

\section{Preliminary results} \label{Preliminary results}

In this section we present various preliminary results. When it is not specified they are assumed to hold for any $\gamma \in \R$, $a \in \R$, $p >2$ and any $c>0.$
\medskip

As already indicated, following \cite{CiWe,Stubbe}, we shall work in the Hilbert space
$$
X:= \{u \in H^1(\R^2)\::\: |u|_*^2 < \infty\} $$
where
$$
|u|_*^2:=  \int_{\R^2} \log(1+|x|)u^2(x)\,dx,
$$ with $X$ endowed with the norm given by 
$ \|u\|_X^2:= \|u\|^2+ |u|_*^2.
$
As in \cite{CiWe} we introduce the symmetric bilinear forms
\begin{align*}
(u,v) & \mapsto B_1(u,v)= \int_{\R^2} \int_{\R^2}
\log
(1+|x-y|) u(x)v(y)\,dx dy,\\
(u,v) & \mapsto B_2(u,v)= \int_{\R^2} \int_{\R^2}
\log \Bigl(1+\frac{1}{|x-y|}\Bigr) u(x)v(y)\,dx dy,\\
(u,v)& \mapsto B_0(u,v)=B_1(u,v)-B_2(u,v)= \int_{\R^2} \int_{\R^2}
\log(|x-y|)  u(x)v(y)\,dx dy,
\end{align*}
and we define on $X$ the associated functionals
\begin{align*}
V_1(u)=B_1(u^2,u^2)= \int_{\R^2} \int_{\R^2}
\log
(1+|x-y|) u^2(x)u^2(y)\,dx dy,\\
V_2(u)=B_2(u^2,u^2)= \int_{\R^2} \int_{\R^2}
\log
\Bigl(1+\frac{1}{|x-y|}\Bigr) u^2(x)u^2(y)\,dx dy.
\end{align*}
Note that $V(u) = V_1(u) - V_2(u)$. 
\smallskip
We shall  use the following results from \cite{CiWe}.

\begin{lem}\cite[Lemma 2.2]{CiWe}
\label{sec:comp-cond}
$ $
\begin{itemize}
\item[(i)] The space $X$ is compactly embedded in $L^s(\R^2)$ for all
  $s \in [2,\infty)$.
\item[(ii)] The functionals $V,V_1,V_2$ and $F$ are of class $C^1$
  on $X$.\\ Moreover, $V_i'(u)v = 4 B_i(u^2,uv)$ for $u,v \in X$ and $i=1,2$.
 \item[(iii)] $V_2$ is continuous (in fact continuously
   differentiable) on $L^{\frac{8}{3}}(\R^2)$.
\item[(iv)] $V_1$ is weakly lower semicontinuous on $H^1(\R^2)$.
\end{itemize}
\end{lem}

\begin{lem}\cite[Lemma 2.1]{CiWe}
\label{lem_Cingo_2.1}
Let $(u_n)$ be a sequence in $L^2(\R^2)$ such that $u_n \overset{L^2(\R^2)}{\longrightarrow} u\neq 0$ pointwise a.e. on $\R^2$. Moreover, let
$(v_n)$ be a bounded sequence in $L^2(\R^2)$ such that
\begin{equation}
  \label{eq:19}
\sup \limits_{n \in \N} B_1(u_n^2,v_n^2)<\infty.
\end{equation}
Then there exists $n_0 \in \N$ and $C>0$ such that $|v_n|_*<C$ for $n \ge n_0$.\\
If, moreover,
\begin{equation}
  \label{eq:18}
B_1(u_n^2,v_n^2) \to 0 \quad \text{and}\quad \|v_n \|_2 \to 0\qquad \text{as $n \to \infty$,}
\end{equation}
then
\begin{equation}
  \label{eq:14}
|v_n|_* \to 0 \qquad \text{as $n \to \infty$.}
\end{equation}
\end{lem}

\begin{lem} \cite[Lemma 2.6]{CiWe} \label{lem_Cingo_2.6}
Let $(u_n),(v_n),(w_n)$ be bounded sequences in $X$ such that $u_n \overset{X}{\rightharpoonup} u$ weakly in $X$.
Then, for every $z\in X$, we have 
$$B_1(v_n\,w_n, z(u_n-u))\rightarrow 0. $$
\end{lem}

From Lemmas \ref{lem_Cingo_2.1} and \ref{lem_Cingo_2.6} we obtain

\begin{lem}\label{V_1}
Let $(u_n) \subset S(c)$ be such that
$u_n \overset{X}{\rightharpoonup} u,$
$ u_n \overset{H}{\rightarrow} u$
and 
$V_1(u_n)\rightarrow V_1(u).$
Then $ u_n \overset{X}{\rightarrow} u.$
\end{lem}

\begin{proof}
In order to show that ${u_n}\rightarrow u$ in X, we have to prove that
$|{u_n}-u|_*\rightarrow 0. $
Since $u_n \overset{H}{\rightarrow} u$, then
${u_n}\rightarrow u$ in $L^2(\R^2)$ with $u\neq 0$. Hence, by Lemma \ref{lem_Cingo_2.1}, 
we actually only need to prove that
$$B_1({u_n}^2,({u_n}-u)^2)\rightarrow 0. $$ But we have 
$$B_1({u_n}^2,({u_n}-u)^2)=V_1({u_n})-2B_1({u_n}^2,({u_n}-u)u)-B_1({u_n}^2,u^2).$$
Since $({u_n})$ is bounded in X and ${u_n}\rightharpoonup u$ in X, 
we know from Lemma \ref{lem_Cingo_2.6} that
$$B_1({u_n}^2,({u_n}-u)u)\rightarrow 0.$$
Hence,
$$\varlimsup B_1({u_n}^2,({u_n}-u)^2)\leq \varlimsup V_1({u_n})-\varliminf B_1({u_n}^2,u^2)\leq \varlimsup V_1({u_n})-V_1(u)$$
by Fatou's Lemma since we may assume that ${u_n}\rightarrow u$ pointwise almost everywhere in $\R^2$.
Since $V_1({u_n})\rightarrow V_1(u)$ and $B_1({u_n}^2,({u_n}-u)^2)\geq 0$, we conclude that $B_1({u_n}^2,({u_n}-u)^2)\rightarrow 0$. Whence the result.
\end{proof}

Our next two lemmas explore the links between the compactness of a sequence in $L^2(\R^2)$ and the boundedness on the functional $V_1$. 

\begin{lem}
\label{lem_vanishing_V1}
Let $(u_n) \subset S(c)$ and assume the existence of $\varepsilon\in (0,c)$ such that for all $R>0$, we have
$$ \varliminf \underset{x\in \mathbb{R}^2}{\sup}\int_{B(x,R)} {u_n}^2 \leq c-\varepsilon.$$
Then,
$$\varlimsup V_1(u_n)=+\infty. $$
\end{lem}

\begin{proof}
Let $R>0$. We denote $$\varepsilon_n= \underset{x\in \mathbb{R}^2}{\sup}\int_{B(x,R)} {u_n}^2.$$
Up to a subsequence, we can assume that 
$ \lim \varepsilon_n\leq c-\varepsilon $. Then, 
\begin{align*}
V_1(u_n)&\geq \iint_{|x-y|\geq R} {u_n}^2(x){u_n}^2(y)\log(1+|x-y|)dxdy\\
&\geq \log(1+R)\left[c^2- 
\iint_{|x-y|\leq R} {u_n}^2(x){u_n}^2(y)dxdy\right].
\end{align*}
But $$\iint_{|x-y|\leq R} {u_n}^2(x){u_n}^2(y)=\int_{\R^2} {u_n}^2(x) \int_{B(x,R)} {u_n}^2(y)\leq c\; \varepsilon_n,$$
hence,
$$V_1(u_n)\geq \log(1+R)\;c\;[c-\varepsilon_n]\geq \log(1+R)\; c\;\frac{\varepsilon}{2}$$
for n large enough, which implies
$$\varlimsup V_1(u_n)\geq \log(1+R)\; c\;\frac{\varepsilon}{2}.$$
Letting $R$ go to infinity, we get the result.
\end{proof}

As a consequence of Lemma \ref{lem_vanishing_V1}, we obtain,

\begin{lem}
\label{lem_V1(u_n)_bornee}
Let $(u_n)$ be a sequence of $S(c)$ such that $\left(V_1(u_n)\right)$ is bounded.\\ Then there exists a subsequence of $(u_n)$ which, up to translation, converges to $u$ in $L^2(\R^2)$. \\More precisely, for all $k\geq 1$, there exist $n_k\rightarrow  \infty$ and $x_{k}\in \R^2$ such that $u_{n_k}(\cdot -x_{k})\rightarrow u$ strongly in $L^2(\R^2)$.

In addition if the sequence $(u_n)$ consists of radial functions than necessarily the sequence $(x_k) \subset \R^2$ is bounded.
\end{lem}

\begin{proof}
Since $\left(V_1(u_n)\right)$ is bounded, we deduce from Lemma \ref{lem_vanishing_V1} that  for all $k\geq 1$, there exist $R_k>0$, $n_k\rightarrow  \infty$ and $x_{k}\in \R^2$ such that 
$$ \int_{B(x_{k},R_k)}{u_{n_k}}^2>c-\frac{1}{k}.$$
Let us set $$v_k=u_{n_k}(\cdot+x_{k}).$$
Since $v_k\in S(c)$, we may assume that, up to subsequence,  $v_k \rightharpoonup v$ weakly in $L^2(\R^2)$.
Moreover, since for all $k\geq 1$, $\int_{\R^2}{v_k}^2\geq \int_{B(0,R_k)}{v_k}^2>c-\frac{1}{k}$, then
$$\lim ||v_k||_2^2 =c.$$ 
Hence, $v_k \rightarrow v$ strongly in $L^2(\R^2)$.

Now if  $(u_n)$ is a sequence of radially symmetric functions we claim that necessarily $(x_k) \subset \R^2$ is bounded. Indeed by Lemma \ref{lem_vanishing_V1} we can fix a $R>0$ such that
 $$ \varliminf \underset{x\in \mathbb{R}^2}{\sup}\int_{B(x,R)} {u_n}^2 \geq \frac{3}{4}c.$$  Then, using  the definition of the supremum and the fact that each $u_n$ is radial if we assume that $(x_k)$ is unbounded we  can find a $n_0 \in \N$ where $x_{n_0}$ satisfies $|x_{n_0}| >R$ such that $$\int_{B(x_{n_0},R)} {u_n}^2 \geq \frac{2}{3}c\quad \mbox{and} \quad  \int_{B(-x_{n_0},R)} {u_n}^2 \geq \frac{2}{3}c,$$
providing a contradiction.
\end{proof}

Our next result establish, under general assumptions, a Pohozaev identity satisfies by the critical points of $F$. A previous, less general version was derived in 
\cite[Lemma 2.4]{DuWe} and we make used of ingredients introduced there in our proof. Note however that we do not make use of the exponential decay of the solutions which is likely not available under our more general assumptions.  As a consequence of this Pohozaev identify any critical point $u \in X$ of $F$ satisfies $Q(u)=0$ and this property will proved crucial in Lemma \ref{convergence}.

			\begin{lem}
\label{lem_Q=0}
Any weak solution $u \in X$ to
\begin{equation}\label{eqt}
- \Delta u + \lambda u + \gamma (log(|\cdot| * |u|^2) u = a |u|^{p-2}u
\end{equation}
where $\lambda \in \R$, $\gamma \in \R$ and $p > 2$, satisfies the Pohozaev identity
\begin{align}
\lambda \int_{\R^2} |u(x)|^2 dx & + \gamma \int_{\R^2}\int_{\R^2}log(|x-y|) |u(x)|^2 |u(y)|^2 dx dy \nonumber \\
&+  \frac{\gamma}{4}\Bigl( \int_{\R^2}|u(x)|^2 dx \Bigr)^2  - \frac{2a}{p}\int_{\R^2}|u(x)|^p dx =0.
\label{PL}
\end{align}
As a consequence it
satisfies $Q(u)=0$ where $Q$ is defined in (\ref{defQ}).
\end{lem}

\begin{proof}
Since $u \in X \subset H^1(\R^2)$, standard elliptic regularity theory yields that $u \in W_{loc}^{2,p}(\R^2)$ for every $p \in [1, \infty)$ and that 
$u \in C^2(\R^2)$. In \cite[Proposition 2.3]{CiWe} it is proved that the function $w : \R^2 \mapsto \R$ given by $w(x)= \int_{\R^2}\log|x-y| u^2(y)\,dy$ is of class $C^3$ on $\R^2$ and satisfies 
$$
 w(x) - \|u\|_2^2  \log |x|
\to 0 \quad \text{as $|x| \to \infty$.}
$$
Since $u \in X$ this implies in particular that $w u^2 \in L^1(\R^2)$.  For notational convenience, let us introduce the functions
$$g(s) = a |s|^{p-2} s - \lambda s \quad \mbox{and} \quad G(s) = \int_0^s g(t) dt = \frac{a |s|^p}{p} - \frac{\lambda s^2}{2}$$ which belong to $C^1(\R)$, since $p>2$. First, following \cite[Proposition 1]{BeLi1}, we multiply the equation (\ref{eqt}) by $x \cdot \nabla u$ and integrate by parts to get a Pohozaev type identity on a ball $B_R(0) : = \{x \in \R^2 \, | \, |x| < R \}$. So let $R >0$. Since, 
for any function $u \in C^2(\R^2)$ we have 
$$\Delta u (x \cdot \nabla u) = div \Big(\nabla u (x \cdot \nabla u) - x \frac{|\nabla u|^2}{2}\Big) \quad \mbox{on} \ \R^2,$$
the divergence theorem gives
\begin{equation}\label{2.5}
\int_{B_R(0)} - \Delta u(x \cdot \nabla u) dx = - \frac{1}{R}\int_{\partial B_R(0)}|x \cdot \nabla u|^2 d \sigma + 
\frac{R}{2} \int_{\partial B_R(0)}|\nabla u|^2 d \sigma.
\end{equation}
Similarly, since $g(u) (x \cdot \nabla u) = div( x G(u))-  2 G(u)$ on $\R^2$, we have
\begin{equation}\label{2.6}
\int_{B_R(0)} g(u) (x \cdot \nabla u) dx =  - 2 \int_{B_R(0)}  G(u) dx  + R \int_{\partial B_R(0)} G(u) d \sigma. 
\end{equation}
Moreover, since $w u (x \cdot \nabla u) = \frac{1}{2} \Big( div[x w u^2] - u^2(x \cdot \nabla w) - 2 w u^2 \Big),$ we have
\begin{equation}\label{2.7}
 \gamma \int_{B_R(0)} w u (x \cdot \nabla u) dx = - \frac{\gamma}{2} \int_{B_R(0)} u^2 (x \cdot \nabla w) dx - \gamma \int_{B_R(0)} w u^2 dx + \frac{\gamma R}{2} 
\int_{\partial B_R(0)} w u^2 d \sigma.
\end{equation}
Thus, multiplying (\ref{eqt}) by $x \cdot \nabla u$ and integrating on $B_R(0)$, we deduce from (\ref{2.5})-(\ref{2.7}) that 
\begin{equation}\label{2.8}
\int_{B_R(0)}  \Big( \gamma\frac{ u^2 (x \cdot \nabla w)}{2} + \gamma w u^2 - 2 G(u) \Big) dx = \int_{\partial B_R(0)} \Big( - \frac{|x \cdot \nabla u|^2}{R} + R \Big( \frac{| \nabla u|^2}{2} + \frac{\gamma w u^2}{2} - G(u)\Big) \Big) d \sigma.
\end{equation}
Next, still following \cite[Proposition 1]{BeLi1}, let us prove that the right hand side in (\ref{2.8}) converges to zero for a suitable sequence $R_n \to \infty$, i.e
\begin{equation*}
R_n \int_{\partial B_{R_n}(0)} |f| d \sigma \to 0 \quad \mbox{for the function} \quad x \mapsto f(x) = \frac{|\nabla u|^2}{2} + \frac{\gamma w u^2}{2} - G(u) - \frac{|x \cdot \nabla u|^2}{|x|^2}.
\end{equation*}
Actually it is a direct consequence of the observation that $f \in L^1(\R^2)$. Indeed, if there is no such sequence $(R_n)$, it follows that
$$ \int_{\partial B_R(0)} |f| d \sigma \geq \frac{c}{R} \quad  \mbox{for} \quad R \geq R_0 \quad \mbox{for some constants} \quad c, R_0 >0$$
and then
$$ \int_{\R^2}|f| dx = \int_0^{\infty} dR \int_{\partial B_R(0)}|f| d \sigma \geq c \int_{R_0}^{\infty} \frac{1}{R}\, d R = \infty.$$
The fact that $f \in L^1(\R^2)$ follows directly using that $u \in H^1(\R^2)$ which implies that $|\nabla u|^2$ and $G(u)$ are in
 $L^1(\R^2)$ and from the already observed property that $wu^2 \in L^1(\R^2)$. 

At this point we deduce from (\ref{2.8}) that
\begin{equation}\label{2.10}
\int_{\R^2}  \Big( \frac{ \gamma u^2 (x \cdot \nabla w)}{2} + \gamma w u^2 - 2 G(u)\Big) dx  = \lim_{n \to \infty} R_n \int_{\partial B_{R_n}(0)} |f| d \sigma =0.
\end{equation}
Now using again that $w u^2$ and $G(u)$ belong to $L^1(\R^2)$ we deduce from (\ref{2.10}) that $u^2 (x \cdot \nabla w) \in L^1(\R^2)$. A direct calculation now gives
$$ x \cdot \nabla w(x) = \int_{\R^2} \frac{|x|^2 - x \cdot y}{|x-y|^2}u^2(y) dy, \quad \mbox{for} \quad x \in \R^2,$$
and thus
\begin{align}
\int_{\R^2}u^2 (x \cdot \nabla w) dx & = \int_{\R^2}\int_{\R^2} \frac{|x|^2 - x \cdot y}{|x-y|^2}u^2(x)u^2(y) dy \nonumber \\
&+  \frac{1}{2} \int_{\R^2}\int_{\R^2} \frac{|x|^2 + |y|^2- 2x \cdot y }{|x-y|^2}u^2(x)u^2(y) dy  = \frac{1}{2} \Big( \int_{\R^2}u^2 dx\Big)^2.
\label{2.11}
\end{align}
From (\ref{2.10}) and (\ref{2.11}) we deduce that (\ref{PL}) holds. 

Now multiplying (\ref{eqt}) by $u$ and integrating we get that
\begin{align}
 	\int_{\R^2} |\nabla u(x)|^2 dx & + \lambda \int_{\R^2} |u(x)|^2 dx  + \gamma \int_{\R^2}\int_{\R^2}log(|x-y|) |u(x)|^2 |u(y)|^2 dx dy \nonumber \\ &
 = a \int_{\R^2}|u(x)|^p dx
\label{eq2}
\end{align}
Combining (\ref{PL}) and (\ref{eq2}) it follows that
$$\int_{\R^2}|\nabla u|^2 dx - \frac{a(p-2)}{p}\int_{\R^2}|u(x)|^p dx - \frac{\gamma c^2}{4} =0$$
and thus, by definition, $Q(u)=0$.
\end{proof}

\begin{lem}\label{convergence}
	Let $(u_n) \subset \Lambda(c)$ be a  Palais-Smale sequence for $F$ restricted to $S(c)$ bounded in $X$. Then, up to a subsequence, $u_n \to u$ strongly in $X$. In particular $u$ is a critical point of $F$ restricted to $S(c)$.
	\end{lem}
	
	\begin{proof}
 We claim that there exists a $\lambda \in \R$ such that $(u_n)$ is Palais-Smale sequence for the functional $F(u) + \frac{\lambda}{2}||u||_2$. Indeed since $(u_n) \subset X$ is bounded we know from  \cite[Lemma 3]{BeLi2} (adapted from the unit sphere to $S(c)$), that
$\| dF_{|_{S(c)}} (u_n)\|_{X^*}=o_n(1)$ is equivalent to $\| dF (u_n) - \frac{1}{c}dF(u_n)(u_n) u_n\|_{X^*}=o_n(1) $.  Now letting
\begin{equation}\label{evvv}
			\lambda_n  : =  - \frac{1}{c} dF(u_n)(u_n) = - \frac{1}{c} \Bigl[A( u_n) + \gamma  V( u_n) - a C( u_n) \Bigr], 
			\end{equation}
			since $(u_n) \subset X$ is bounded we deduce that $(\lambda_n) \subset \R$ is bounded. So, up to a subsequence, $\lambda_n \to \lambda \in \R$ as  $n \to  \infty$ and this proves the claim. At this point, using that $(u_n) \subset X$ is bounded and $dF(u_n) + \frac{\lambda}{2}u_n \to 0$ in $X^*$ we shall deduce that $(u_n)$ strongly converges in $X$ to a $u \in X$ which will thus be a critical point of $F$ restricted to $S(c)$.

			Since $(u_n)$ is bounded in $X$ we can assume, passing to a subsequence if necessary, that $u_n \rightharpoonup u$ weakly in $X$ and, see Lemma \ref{sec:comp-cond}(i), that $u_n \to u$ strongly in $L^s(\R^2)$ for $s \in [2, \infty)$. Next we observe that, since for any $\phi \in X$,
			$$ (dF(u_n) + \frac{\lambda}{2}u_n)\phi \to 0$$
			we have that 
			$$dF(u) + \frac{\lambda}{2} u  =0 \quad \mbox{in} \quad X^*.$$
Namely $u$ is solution to (\ref{eqt}) and by Lemma \ref{lem_Q=0} we deduce that $Q(u)=0$. Now observe that, since $(u_n) \subset \Lambda(c)$, we have, using that $Q(u)=0$
$$ 0 = Q(u_n) = A(u_n) - a \frac{p-2}{p}C(u_n) + \frac{\gamma c^2}{4} = A(u) - a \frac{p-2}{p}C(u) + \frac{\gamma c^2}{4}.$$
Since $C(u_n) \to C(u)$ we then necessarily have $A(u_n) \to A(u)$. In particular $u_n \to u$ in $H^1(\R^2)$. Finally we observe that, since $A(u_n) \to A(u)$ and $u_n \to u$ strongly in $L^s(\R^2)$ for $s \in [2, \infty)$,
				\begin{align}\label{zut}
o(1)
        &= (F'(u_n) + \frac{\lambda}{2}u_n) (u_n -u)  \\
        & = o(1) + A(u_n) - A(u) + \frac{\gamma}{4}V'(u_n)(u_n-u) - a \int_{\R^2} |u_n|^{p-2}u_n(u_n -u)  \nonumber\\
        & =  o(1)  + \frac{\gamma}{4}[V_1'(u_n)(u_n-u) - V_2'(u_n)(u_n-u)]
             \nonumber
\end{align}	
where
$$ \Big| 		V_2'(u_n)(u_n-u)\Big| = | B_2(u_n^2, u_n(u_n -u))\Big|  \leq  ||u_n||_{\frac{8}{3}}^3 ||u_n -u||_{\frac{8}{3}} \to 0$$
as $n \to \infty$ and
$$ V_1'(u_n)(u_n -u) = B_1(u_n^2, u_n(u_n-u)) = B_1(u_n^2, (u_n-u)^2) + B_1(u_n^2,u(u_n-u))$$
with $$B_1(u_n^2, u(u_n-u)) \to 0 \quad \mbox{as} \quad n \to \infty$$
by Lemma \ref{lem_Cingo_2.6}. Combining these estimates we obtain that
$$o(1) = o(1) + B_1(u_n^2, (u_n-u)^2), $$
which implies that $B_1(u_n^2, (u_n-u)^2) \to 0$ as $n \to \infty$. Hence by Lemma \ref{lem_Cingo_2.1}, $|u_n - u|_* \to 0$ as $n \to \infty.$ We conclude that $||u_n - u||_X \to 0$ as $n \to \infty$ as claimed. This ends the proof of the lemma.
\end{proof}

Finally, for future reference, note that using the Gagliardo-Nirenberg inequality
\begin{equation}\label{GN}
||u||_p \leq K_{GN}^{\frac{1}{p}} ||\nabla u||_2^{\beta} ||u||_2^{1- \beta} \mbox{where} \quad \beta = 2 \Big(\frac{1}{2}- \frac{1}{p}\Big)
\end{equation}
			 we obtain that
			\begin{equation}
\label{minoration_norme_u_L^p}
C(u)=\|u\|_p^p\leq K_{GN}\; A(u)^{\frac{p}{2}-1}\;c. 
\end{equation}
Also by (2.2) in \cite{CiWe}
			$$|V_2(u)| \leq C_0 ||u||_{8/3}^4$$
			and using (\ref{GN}) with $p=\frac{8}{3}$ we get that for some best constant $K >0,$ for all $u \in H$,
			\begin{equation}
\label{minoration_V2}
|V_2(u)|\leq K \; \sqrt{A(u)}\;c^{\frac{3}{2}}.
\end{equation}

\section{The case $\gamma>0$}\label{gammapositif}

Throughout this section we assume that $\gamma >0$.

\begin{lem}\label{X-bound}
Assume that $\gamma >0$ and let $(u_n) \subset S(c)$ be a bounded sequence in $H$  such that $F(u_n) \leq d $ for some $d \in \R$. Then 
there exists a sequence
 $(x_n) \subset \R^2$ such that $\tilde{u_n}=u_n(\cdot-x_n) $ has a subsequence converging weakly in X. 

If in addition $(u_n) \subset S(c)$ consists in radially symmetric functions, the sequence $(x_n) \subset \R^2$ is bounded.
\end{lem}

\begin{proof}
Since $A(u_n)$ is bounded, hence $V_2(u_n)$ and $\|u_n\|_p^p$ are also bounded by \eqref{minoration_V2} and  \eqref{minoration_norme_u_L^p}. 
Now since $F(u_n)$ is bounded, then $V_1(u_n)$ also. We then deduce by Lemma \ref{lem_V1(u_n)_bornee} the existence of $(x_n) \subset \R^2$, which is bounded if $(u_n) \subset S(c)$ consists in radially symmetric functions, such that, if we denote
 $$\tilde{u_n}=u_n(\cdot-x_n) ,$$
 then, up to a subsequence,   
$$\tilde{u_n}\overset{L^2(\R^2)}{\longrightarrow}u. $$
Now, by Lemma \ref{lem_Cingo_2.1},
since $V_1(\tilde{u_n})=V_1(u_n)$ is bounded and $\tilde{u_n}\rightarrow u\neq 0$ 
in $L^2(\R^2)$, we deduce that $|u_n|_*$ is bounded, so $(\tilde{u_n})$ is bounded in X. Since X is a Hilbert space, then, up to subsequence, we may assume that
$\tilde{u_n} \rightharpoonup u.$
\end{proof}

\begin{lem}
\label{lem_cv_faible-->cv_forte}
Assume that $\gamma >0$ and let $(u_n) \subset S(c)$ be such that
$ u_n \overset{X}{\rightharpoonup} u$. 
Then $F(u) \leq \varliminf F(u_n).$
If moreover
$F(u_n)\rightarrow F(u)$ then 
$u_n\overset{X}{\longrightarrow} u\in S(c).$
\end{lem}

\begin{proof}
Since $X$ is compactly embedded in $L^s(\R^2)$ for all $x \in [2, \infty)$, see Lemma \ref{sec:comp-cond}(i), we deduce that $u \in S(c)$, $C(u_n) \to C(u)$ and by the continuity of $V_2$ on $L^{\frac{8}{3}}(\R^2)$, see Lemma \ref{sec:comp-cond}(iii), that $V_2(u_n) \to V_2(u)$. The fact that
$$F(u) \leq \varliminf F(u_n)$$
then follows from the weak lower semicontinuity of $u \mapsto V_1(u)$ on $H$ (and thus on $X$), see Lemma \ref{sec:comp-cond}(iv), and of $u \mapsto A(u)$ on $H$.

Now assume that $F(u_n)\rightarrow F(u).$ We shall see that
$A({u_n})\rightarrow A(u)$
and
$V_1({u_n})\rightarrow V_1(u),$
which in particular implies  that ${u_n}\rightarrow u$ strongly in H and then, by  Lemma \ref{V_1} that
$u_n\overset{X}{\longrightarrow} u.$ 
 Indeed, considering $F({u_n})-F(u)$, we get
\begin{equation}
\label{equation_A_V1}
\frac{1}{2}\left[A({u_n})-A(u) \right]+ \frac{\gamma}{4}\left[V_1({u_n})-V_1(u) \right]=o(1).
\end{equation}
Hence, taking the liminf, we get
$$\frac{1}{2}\left[\varliminf A({u_n})-A(u) \right]+ \frac{\gamma}{4}\left[\varliminf V_1({u_n})-V_1(u) \right]\leq 0.$$
Using the lower semicontinuity of A (resp. $V_1$) with respect to the weak $H^1$ (resp. $X$) convergence, we then deduce that
$$ \varliminf A({u_n})=A(u)  \quad \mbox{and} \quad \varliminf V_1({u_n})=V_1(u).$$
Taking the limsup in \eqref{equation_A_V1}, we get the desired result.
\end{proof}

\subsection{Proof of Theorem \ref{minimization1}.}\label{subsection31}

This subsection is mainly devoted to the proof of Theorem \ref{minimization1}. We start with the following lemma.

\begin{lem}
\label{lem_fonctionelle_minoree}
Under the assumption of Theorem \ref{minimization1}, 
$ m= \underset{S(c)}{\inf} F>-\infty.$
\end{lem}

\begin{proof}
\underline{First case: } $a\leq 0$ and $p>2$.\\
Then, since $V_1\geq 0$ and $a\leq 0$, we have for all $u\in X$, using (\ref{minoration_V2})
\begin{equation}
\label{minoration_F_cas1a}
F(u)\geq \frac{1}{2}A(u)-\frac{\gamma}{4}V_2(u)\geq 
\frac{1}{2}A(u)-\frac{\gamma}{4} K \; \sqrt{A(u)}\;c^{\frac{3}{2}},
\end{equation}
whence the result.\\
\underline{Second case: } $a> 0$ and $p<4$.\\
Then, for all $u\in X$, we have, using (\ref{minoration_V2}) and (\ref{minoration_norme_u_L^p})
\begin{equation}
\label{minoration_F_cas1b}
F(u)\geq 
\frac{1}{2}A(u)-\frac{\gamma}{4} K \; \sqrt{A(u)}\;c^{\frac{3}{2}}-\frac{a}{p}K_{GN}\; A(u)^{\frac{p}{2}-1}\;c,
\end{equation}
whence the result since $\frac{p}{2}-1<1$.\\
\underline{Third case: } $a> 0$, $p=4$ and $c < \displaystyle \frac{2}{a K_{GN}}$.\\
From (\ref{minoration_F_cas1b}) we get that
$$
F(u) \leq \Big(\frac{1}{2}- \frac{a}{4}K_{GN}c \Big) A(u) - \frac{\gamma}{4} K \; \sqrt{A(u)}\;c^{\frac{3}{2}}$$
and the result follows here also.
\end{proof}

As a consequence of Lemma \ref{lem_fonctionelle_minoree} and of the convergence results of Section \ref{Preliminary results} we can now give

\begin{proof}[Proof of Theorem \ref{minimization1}]
By Lemma \ref{lem_fonctionelle_minoree} we know that $m >- \infty$.
Now let $(u_n)$ be a minimizing sequence for (\ref{definfimum}). 
Since $(F(u_n))$ is bounded from above we deduce from \eqref{minoration_F_cas1a} or \eqref{minoration_F_cas1b}, that $(u_n)$ is bounded in $H$. Thus we deduce from Lemma \ref{X-bound}
 the existence of $(x_n) \subset \R^2$ such that, if we denote
 $\tilde{u_n}=u_n(\cdot-x_n) ,$
 then, up to a subsequence, we may assume that
$\tilde{u_n} \rightharpoonup u.$ Moreover, recording that the embedding $X\subset L^2(\R^2)$ is compact, we have that $u \in S(c)$. At this point, since $F$ is  invariant by translation, we deduce from Lemma \ref{lem_cv_faible-->cv_forte} that $F(u) = m$ and that $\tilde{u_n} \rightarrow u $ in $X$.
\end{proof}

We end this section by observing that setting
$$\Sigma=\{u\in S(c),\; F(u)=m\}, $$
and for any $R>0$,
$$ \Sigma(R)=\Sigma\cap B_X(0,R)$$
we have 

\begin{lem}
There exists $R>0$ such that $\Sigma=\R^ 2*\Sigma(R).$
\end{lem}

\begin{proof}
We argue by contradiction. Assume that for all $n\geq 1$, there exists $u_n\in \Sigma$ such that $u_n\notin \R^ 2*\Sigma(n)$. Reasoning as in the proof of 
Theorem \ref{minimization1} we deduce that  there exists a sequence $(x_n)$ of $\R^2$ such that 
$$\tilde{u_n}=u_n(\cdot-x_n) $$ has a subsequence bounded in $X$. But by hypothesis, $\tilde{u_n}\notin \Sigma(n)$ so $\|\tilde{u_n}\|_X\geq n$, which is a contradiction. 
\end{proof}

\subsection{Existence of a local minima on $S(c)$.}\label{subsection32}
In this subsection we always assume that $a>0$, $p>4$. We also set 
$$k_0= \frac{(p-2)}{(p-4)}\; \frac{\gamma \,c^ 2}{4}.$$

\begin{lem}\label{natural-bound}
\label{lem_c_0}
Assume that $\gamma >0$, $a >0$ and $p>4$.
If $Q(u)\leq 0$ and $A(u)=k_0$, then 
$c\geq c_0,$
where $c_0$ depends on $p,a$ and $\gamma$ by the following formula :
$$c_0=2\left[\frac{p(p-4)^{\frac{p-4}{2}}}{(p-2)^\frac{p}{2}}\frac{1}{K_{GN}}\frac{1}{a\gamma^{\frac{p-4}{2}}}\right]^ {\frac{1}{p-3}}.$$
As a consequence, if $Q(u)\leq 0$ and $c<c_0$, then
$A(u)\neq k_0.$
\end{lem}

\begin{proof}
Since $Q(u)\leq 0$, we have 
$$A(u)\leq a \frac{p-2}{p} C(u)+\frac{\gamma\,c^2}{4}$$
and  by Gagliardo-Nirenberg inequality (\ref{GN}), since $A(u)=k_0$,
 we deduce
$$\frac{p-2}{p-4}\;\frac{\gamma\,c^2}{4}\leq a \frac{p-2}{p}K_{GN} \left[\frac{p-2}{p-4}\; \frac{\gamma \,c^ 2}{4} \right]^{\frac{p-2}{2}}\;c +\frac{\gamma\,c^2}{4},$$
then
$$\frac{2}{p-4}\;\frac{\gamma\,c^2}{4}\leq a \frac{p-2}{p}K_{GN} \left[\frac{p-2}{p-4}\right]^{\frac{p-2}{2}} \frac{\gamma^{\frac{p-2}{2}}}{2^{p-2}}\; c^{p-1},$$
so
$$2^{p-3} \frac{p(p-4)^{\frac{p-4}{2}}}{(p-2)^\frac{p}{2}}\frac{1}{K_{GN}}\frac{1}{a\gamma^{\frac{p-4}{2}}}\leq c^{p-3},$$
whence the result.
\end{proof}

Now we set
$$A_{k_0}=\{u\in X,\;A(u)\leq k_0\}$$
and we define 
\begin{equation}\label{minimum-local}
m_l=\underset{S(c)\cap A_{k_0}}{\inf}F.
\end{equation}

\begin{thm}\label{minimization local}
Let $\gamma >0$, $a >0$ and $p>4$.
Assume that $c<c_0,$ then any minimizing sequence for $m_l$ defined in (\ref{minimum-local})
has, up to translations, a subsequence converging strongly in X. In particular the  infimum  is achieved. Also any minimizer of (\ref{minimum-local}) is a critical point of $F$ on $S(c)$.
\end{thm}

\begin{proof}
Let $(u_n)$ be a minimizing sequence for (\ref{minimum-local}). Reasoning exactly as in the proof of Theorem \ref{minimization1} we see that there exists a sequence of ($x_n)\subset \R^2$ such that,
 $\tilde{u_n}=u_n(\cdot-x_n) ,$
converges strongly towards a $u \in X$. Obviously $u \in A_{k_0}$ and $F(u) = m_l$. Thus to end the proof it just remains to show that $u$ satisfies   $A(u) <k_0$.

Let us assume by contradiction that $A(u) =k_0$. Then we see directly from Lemma \ref{natural-bound} that necessarily $Q(u) >0.$  But then we consider $u^{t_0}$ with $t_0 <1$ close to $1$. Recording (\ref{scalings}) and (\ref{LinkF-Q}) it follows that $u^{t_0} \in A_{k_0}$ and $F(u^{t_0}) < F(u)$ providing a contradiction. This ends the proof.
\end{proof}

\subsection{Proof of Theorem \ref{infreachxxxx}.}\label{subsection33}

In this subsection we start to be interested in the multiplicity of solutions. We shall always assume that $a>0$ and $p >4$. For any $u \in S(c)$ we denote  $g_u : (0, \infty) \to X$ the function defined by 
		$$g_u(t)= F(u^t)= \frac{t^2}{2} A(u) + \frac{\gamma}{4} V(u) - \frac{\gamma c^2}{4} \log t 
		-\frac{a t^{p-2}}{p} C(u)
		$$ 
		where $u^t(x)= t u(t x)$ for all $x \in \R^2$. Clearly $g_u$ is $C^2$ on $(0,  \infty)$ and we obviously have
		$$
		\Lambda(c) := \{u \in S(c) \  | \ Q(u) =0\} = \{u \in S(c) \ | \ g_u'(1)=0\}.
		$$
	\begin{lem}
			For any $u \in S(c)$, a value $s\in \R$ is critical for $g_u(t)$ if and only if $u^s \in \Lambda(c)$.
		\end{lem}
		\begin{proof}
			Fix $u \in S(c)$. We have 
			$$
			g_u'(t)= \frac{1}{t} \bigl( t^2 A(u) - \frac{\gamma c^2}{4} - a \frac{(p-2)}{p} t^{p-2} C(u) \bigr).
			$$
			Therefore $s >0$ is a critical value for $g_u$ if and only if 
			$$
			s^2 A(u) - \frac{\gamma c^2}{4} - a \frac{(p-2)}{p} s^{p-2} C(u)=0
			$$
			which means
			$$ 
			A(u^s) - \frac{\gamma c^2}{4} - a \frac{(p-2)}{p}C(u^s)=0
			$$
			namely $g'_{u^s}(1)=0$ and thus  $u^s \in \Lambda(c)$.
		\end{proof}

		Now we prove the following lemmas.

	\begin{lem}\label{submanifold1}
			If $c <c_0$, then $\Lambda(c)$ is a submanifold of codimension $2$ of $X$ and a submanifold of codimension $1$ in $S(c)$.
		\end{lem}
		
\begin{proof}
By definition, $u \in \Lambda(c)$ if and only if $G(u):=\|u\|_2^2 -c =0$ and $Q(u)=0.$ It is easy to check that $G, Q$ are of $C^1$ class. Hence we only have to prove that for any $u \in \Lambda(c)$,
$$
(dG(u), dQ(u)): X \rightarrow \R^2 \,\, \mbox{is surjective}.
$$
If this failed, we would have that $dG(u)$ and $dQ(u)$ are linearly dependent, which implies that there exists a $\nu \in \R$ such that for any $\varphi \in X$,
$$
 2 \int_{\R^N} \nabla u \cdot \nabla \varphi \, dx
- a(p-2) \int_{\R^N}|u|^{p-2}u \varphi \, dx=  2 \nu \int_{\R^N}u \varphi\,dx,
$$
namely that $u$ solves
$$
- \Delta u  - a\frac{(p-2)}{2} |u|^{p-2}u = \nu u.
$$
At this point from Lemma \ref{lem_Q=0} we deduce that
$$
A(u) = \frac{a(p-2)^2}{2p}C(u)
$$
and then, since  $Q(u) =0$ we obtain that $A(u) =k_0$ which contradicts Lemma \ref{natural-bound}.
\end{proof}
	
	\begin{lem}\label{nondegenerate}
\label{lem_A(u)=k_0_1}
Let $u\in S(c)$ such that $Q(u)=0$ and $\left.\frac{d}{dt}\right|_{t=1}\;Q(u^t)=0$.
Then $ A(u)=k_0.$
\end{lem}

\begin{proof}
First, a simple computation shows that 
$$\left.\frac{d}{dt}\right|_{t=1}\;Q(u^t)=2A(u)-a\frac{(p-2)^2}{p}\; C(u).$$
So by hypothesis, $$a\frac{p-2}{p}\;C(u)=\frac{2}{p-2}\; A(u).$$
But we also know that $Q(u)=A(u)-a\;\frac{p-2}{p}\; C(u)-\gamma\;\frac{c^ 2}{4}=0$, so
$$\left(1-\frac{2}{p-2}\right)\; A(u)=\gamma\;\frac{c^ 2}{4},$$ i.e.
$A(u)=k_0.$
\end{proof}	
		
		Let us denote 
		$$
		\Lambda^+(c) = \{ u \in S(c) \ |\ g'_u(1)=0, \ g''_u(1) > 0 \}.
		$$
		$$
		\Lambda^-(c) = \{ u \in S(c) \ |\ g'_u(1)=0, \ g''_u(1) < 0 \}.
		$$
		$$
		\Lambda^0(c) = \{ u \in S(c) \ |\ g'_u(1)=0, \ g''_u(1) = 0 \}.
		$$
Observe that  $\Lambda^0(c) = \emptyset$  when $c < c_0$ by Lemmas \ref{natural-bound} and	\ref{nondegenerate}.
		\bigskip
		\begin{lem}\label{intersection}
			Let $c < c_0$. For any $u \in S(c)$, there exists 
			\begin{enumerate}
			\item a unique $s_u^+ >0$ such that $ u^{s_u^+} \in  \Lambda^+(c)$.
			Such $s_u^+$ is a strict local minimum point for $g_u$.
			\item a unique $s_u^- >0$ such that  $u^{s_u^-}  \in  \Lambda^-(c)$.
			Such $s_u^-$ is a strict local maximum point for $g_u$.
			\end{enumerate} 
		\end{lem}
		\begin{proof}
			Fix $u \in S(c)$ with $c < c_0$. Let $t^* = \bigl[\frac{2p A(u)}{a(p-2)^2 C(u)} \bigr]^{1/(p-4)}$,  which means
			$$
			2 ({t^*})^2 A(u) = \frac{a(p-2)^2}{p} ({t^*})^{p-2} C(u)
			$$
			namely
			$$
			2 A(u^{t^*}) = \frac{a(p-2)^2}{p} C(u^{t^*}).
			$$
			It follows that 
			\begin{equation}\label{prima}
			2 A(u^t) > \frac{a(p-2)^2}{p} C(u^t), \quad \forall \, 0 < t < t^*
			\end{equation}	
			and 
			\begin{equation}\label{seconda}
			2 A(u^t) <\frac{a(p-2)^2}{p} C(u^t), \quad \forall  t > t^*.
			\end{equation}	
			By $(\ref{prima})$ we have that for any $t \in (0, t^*),$ 
			\begin{align}
g_u'(t) & = \frac{1}{t} 
			\bigl( A(u^t) - \frac{\gamma c^2}{4} - a \frac{(p-2)}{p} C(u^t) \bigr) \nonumber \\
& >  \frac{1}{t}
			\bigl( A(u^t) - \frac{\gamma c^2}{4} -  \frac{2}{(p-2)} A(u^t) \bigr)   = \frac{1}{t}
			\Bigl( \frac{(p-4)}{(p-2)} A(u^t) - \frac{\gamma c^2}{4}  \Bigr). 
\label{key}
\end{align}
			Now we prove that if  $c < c_0 $, then 
			\begin{equation}\label{sso}
			A(u^{t^*}) > \frac{\gamma c^2}{4}\frac{(p-2)}{(p-4)} =k_0. 
			\end{equation}
			In fact, taking into account the Gagliardo-Nirenberg inequality, 
			we have 		
			\begin{align}
(t^*)^2  A(u) \frac{(p-4)}{(p-2)} - \frac{\gamma c^2}{4}  & = \bigl[ \frac{2p A(u)}{a(p-2)^2 C(u)} \bigr]^{\frac{2}{p-4}}
			A(u) \frac{p-4}{p-2} - \frac{\gamma c^2}{4}  \nonumber \\
& =  \frac{\bigl(A(u)\bigr)^{\frac{p-2}{p-4}}}{ \bigl(C(u)\bigr)^{\frac{2}{p-4}} }
			\bigl[ \frac{2 p}{a(p-2)^2} \bigr]^{\frac{2}{p-4}}
			\frac{(p-4)}{(p-2)} - \frac{\gamma c^2}{4} \nonumber \\
& \geq \frac{\bigl(A(u)\bigr)^{\frac{p-2}{p-4}}}{
				\bigl[K_{G N} A(u)^{\frac{p-2}{2}} c \bigr]^{\frac{2}{p-4}}}
			\bigl[ \frac{2 p}{a(p-2)^2} \bigr]^{\frac{2}{(p-4)}}
			\frac{(p-4)}{(p-2)} - \frac{\gamma c^2}{4} \nonumber \\
& \geq c^2  \bigl( c^{\frac{2(3-p)}{(p-4)}}
			\bigl[ \frac{2 p}{a(p-2)^2  K_{G N}} \bigr]^{\frac{2}{(p-4)}}
			\frac{(p-4)}{(p-2)} - \frac{\gamma}{4} \bigr) \nonumber \\
& = c^2  \bigl( c^{\frac{2(3-p)}{(p-4)}}
			\bigl[ \frac{2 p}{a K_{G N}} \bigr]^{\frac{2}{(p-4)}}
			\frac{(p-4)}{(p-2)^{\frac{p}{p-4}}} - \frac{\gamma}{4} \bigr)  > 0
\label{alll}
\end{align}
since
			$$
			c < c_0 = 
			2 \ \Bigl[ \frac{p(p-4)^{\frac{p-4}{2}}}{(p-2)^{p/2}} \frac{1}{K_{G N}}  \frac{1}{a \gamma^{\frac{p-4}{2}}}  \Bigr]^{\frac{1}{p-3}}.
			$$
			It follows that there exists $\delta >0$ such that for any $t \in (t^* - \delta, t^*)$ 
			$$
			A(u^{t}) > k_0. 
			$$
			By $(\ref{key})$ we infer that for any $t \in (t^* - \delta, t^*)$,
			$g_u'(t) >0$ and thus $g_u(t)$ is increasing in $(t^*-\delta, t^*)$.
			
			Taking into account that the function $g_u(t) \to + \infty $ as $t \to 0^+$ and $g_u(t) \to - \infty$ as $t \to + \infty$,
			we conclude that there exists at least a critical point $s_u^+ < t^*$ which is a local minimum point of $g_u$ and a critical point $s_u^- > t^*$ which is a local maximum point of $g_u$. We first consider $s_u^- >0$. Since  $s_u^- >t^*$, from  $(\ref{seconda})$ we derive that
			\begin{equation}\label{dis}
			2 (s_{u}^-)^2 A(u)  - a \frac{(p-2)^2}{p} (s_{u}^-)^{p-2} C(u) <0.
			\end{equation}
			Moreover from $(\ref{dis})$ and the fact that $g_u'(s_{u}^-)=0$, we derive that 	 	
			\begin{eqnarray}\label{pos}
			g_u''(s_{u}^-)=  \frac{1}{(s_{u}^-)^2} 
			\bigl(A(u^{s_{u}^-}) + \frac{\gamma c^2}{4} - a \frac{(p-2)(p-3)}{p} C(u^{s_{u}^-}) \bigr) \nonumber \\ 
			= \frac{1}{(s_{u}^-)^2} 
			\bigl( 2 (s_{u}^-)^2 A(u)  - a \frac{(p-2)^2}{p} (s_{u}^-)^{p-2} C(u) \bigr) < 0.
			\end{eqnarray}
			Therefore $s_{u}^-$ is a strict maximum point for $g_u$ 	
			and $u^{s_{u}^-} \in \Lambda^-(c)$.
			
			We have to show that $s_{u}^-$ is unique.
			By contradiction we assume that 
			there exists  $z_u>0$ an other critical point of $g_u$ which is a local maximum point.

			Firstly we observe that if $0< z_u <t^*$, then from  $g_u'(z_u)=0$ and $(\ref{prima})$ it results 
			\begin{eqnarray}\label{nib}
			g_u''(z_u)
			= \frac{1}{z_u^2} 
			\bigl( 2 z^2_u A(u)  - a \frac{(p-2)^2}{p} z_u^{p-2} C(u) \bigr) > 0
			\end{eqnarray}
			which is a contradiction.
			This implies that $z_u> t^*$ and thus arguing as before we have 
			$g_u''(z_u) < 0$ namely  $u^{z_u} \in \Lambda^-(c)$. 
			We derive the existence of an other critical point $\theta_u >t^*$, which is a local minima for $g_u$.
			Taking into account $(\ref{seconda})$, we again deduce
			$g_u''(\theta_u) < 0$, which is a contradiction.
			Therefore the point $s_u$ is unique.
			
			Now a direct adaptation of the argument used for $s_u^- >0$ leads to conclude that $s_u^+ >0$ is the unique local minimum point for $g_u$.
		\end{proof}

		\begin{lem}\label{regularite}
			Let $c < c_0$. The maps $u \in S(c) \mapsto s_{u^+} \in \R$ and $u \in S(c) \mapsto s_{u^-}\in \R$ are of class $C^1$.
		\end{lem}
		\begin{proof}
				It is a direct application of the Implicit Function Theorem on the $C^1$ function
				$\Psi : \R \times S(c)  \to \R$, defined by  				
				 $\Psi(s,u)= g'_u(s)$, taking into account that  $\Psi(s^\pm_u,u)=0$, $\partial_s \Psi(s^+_u,u) = g_u''(s^+_u) >0$, 
$\partial_s \Psi(s^+_u,u) = g_u''(s^-_u) <0$ and $\Lambda^0(c)= \emptyset$.		
		\end{proof}

\begin{lem}\label{coercive}
$F$ restricted to $\Lambda(c)$ is coercive on $H$ and bounded from below by a positive constant. 
\end{lem}
\begin{proof}
Firstly we observe that if $u \in \Lambda(c)$,	then  
\begin{equation}\label{3}
C(u) = \frac{p}{a(p-2)} \Big[ A(u) - \frac{\gamma c^2}{4}\Big].
\end{equation}
Taking into account that $\gamma V_1(u) \geq 0$ and (\ref{minoration_V2}), we get that
\begin{equation*}
F(u) \geq \frac{1}{2} A(u) - \frac{\gamma c^{3/2}}{4} A(u)^{\frac{1}{2}} - \frac{1}{p-2}\Big[A(u) - \frac{\gamma c^2}{4}\Big] \geq \Big[\frac{1}{2} - \frac{1}{p-2}\Big] A(u) - \frac{\gamma c^{3/2}}{4} A(u)^{\frac{1}{2}}.
\end{equation*}
Since $p >4$, this concludes the proof.
\end{proof}

	In view of Lemma \ref{coercive} we can define
	$$\gamma^+(c) := \inf_{\Lambda^+(c)} F(u) \quad \mbox{and} \quad \gamma^-(c) := \inf_{\Lambda^-(c)} F(u).$$
	
	Aiming to prove Theorem \ref{infreachxxxx} we shall establish the existence of a Palais-Smale sequence $(u_n) \subset \Lambda^{+}(c)$ (respectively $(u_n) \subset \Lambda^{-}(c)$)  for $F$ restricted to $S(c)$. Our arguments are inspired from \cite{BaSo2}.

We start by recalling the following definition \cite[Definition 3.1]{Gh}.
\begin{definition}
Let $B$ be a closed subset of a metric space $Y$. We say that a class $\mathcal{G}$ of compact subsets of $Y$ is a homotopy stable family with closed boundary $B$ provided
\begin{enumerate}
\item every set in $\mathcal{G}$ contains $B$;
\item for any $A\in \mathcal{G}$ and any $\eta \in C([0,1]\times Y, Y)$ satisfying $\eta (t,x)=x$ for all $(t,x)\in (\{0\}\times Y)\cup ([0,1]\times B)$, we have $\eta (\{1\} \times A)\in \mathcal{G}$.
\end{enumerate}
\end{definition}
We explicitly observe that $B=\varnothing$ is admissible. Now we define the two functionals 
$$I^+ : S(c) \mapsto \R \quad \mbox{by} \quad I^+(u) = F(u^{s_{u}^+})$$
and  $$ I^- : S(c) \mapsto \R \quad \mbox{by} \quad I^-(u) = F(u^{s_{u}^-}).$$
Note that since the maps $u \mapsto s_{u^{+}}$  and $u \mapsto s_{u^{-}}$ are of class $C^1$, see Lemma \ref{regularite}, the functionals $I^+$ and $I^-$ are of class $C^1$.
\begin{lem}\label{isomorphism}
The maps $T_u S(c)\rightarrow T_{ u^{s^+_{u}}} S(c)$ defined by  $\psi \rightarrow  \psi^{s^+_{u}}$  and
$T_u S(c)\rightarrow T_{ u^{s^-_{u}}} S(c)$ defined by  $\psi \rightarrow \psi^{s^-_{u}} $
 are isomorphisms. 
\end{lem}

\begin{proof}
We give a proof of the first statement and to shorten the notation we set $s = s^+_{u}$ and $ u^s =u^{s^+_{u}}$. For $\psi \in T_u S(c)$ we have
$$\int_{\R^2} u^s(x) \psi^s(x) \, dx = \int_{\R^2} s u(s x)  s \psi(s x) \, dx = \int_{\R^2} u(y) \psi(y) \, dy =0. $$
As a consequence, $\psi^s \in T_{u^s}S(c)$ and the map is well defined. Clearly it is linear and the rest of the proof is standard, see for example \cite[Lemma 3.6]{BaSo2}. 
\end{proof}

\begin{lem}\label{transform}
We have that $dI^+(u)[\psi]=dF(u^{s^+_u}) [\psi^{s^+_{u}}]$ and $dI^-(u)[\psi]=dF(u^{s^-_u}) [\varphi^{s^-_{u}}]$ for any $u\in S(c)$ and $\psi \in T_u S(c)$.
\end{lem}

\begin{proof}
We give the proof for $I^+$, we set here $s_u = s_{u}^+$ and $\psi^{s_u} = \psi^{s_{u}^+}$. Our proof is inspired by \cite[Lemma 3.2]{BaSo1}. Let $\psi \in T_uS(c)$. Then $\psi = \gamma'(0)$ where $ \gamma : (- \varepsilon, \varepsilon ) \mapsto S(c)$ is a $C^1$-curve  with $\gamma(0) = u$. We consider the incremental quotient
\begin{equation}\label{limite}
\frac{I^+(\gamma(t)) - I^+(\gamma(0))}{t} = \frac{F(\gamma(t)^{s_t}) - F(\gamma(0)^{s_0})}{t}
\end{equation}
where $s_t := s_{\gamma(t)}$ (notice that $s_0 = s_u$). Recalling that $s_t$ is a strict local minimum of $s \mapsto F( u^s)$ and using that $u \mapsto s_u^+$ is continuous, see Lemma \ref{regularite}, we get for $|t|$ small
\begin{eqnarray}\label{CU}
			F(\gamma(t)^{s_t}) - F(\gamma(0)^{s_0}) \geq  F(\gamma(t)^{s_t}) - F(\gamma(0)^{s_t})	= \frac{s_t^2}{2}\Big[A(\gamma(t)) - A(\gamma(0)\Big]
			 \nonumber \\ 
			 + \frac{\gamma}{4}\Big[V(\gamma(t)) - V(\gamma(0))\Big] - \frac{a s_t^{p-2}}{p}\Big[C(\gamma(t))- C(\gamma(0))\Big]
			 \nonumber \\ =
			s_t^2 \int_{\R^2}\nabla \gamma(\tau_1 t) \cdot \nabla \gamma^{'}(\tau_1 t) t \, dx   
			+  \gamma \int_{\R^2}\int_{\R^2} log |x-y| (\gamma(\tau_2 t))^2 (x) \gamma(\tau_2 t) (y) \gamma^{'}(\tau_2 t) (y) \, dx dy.
			\nonumber \\ -
			a s_t^{p-2} \int_{\R^2}|\gamma(\tau_3 t)|^{p-2} \gamma(\tau_3 t) \gamma^{'}(\tau_3 t) t \, dx \nonumber
			\end{eqnarray}
for some $\tau_1, \tau_2, \tau_3 \in (0,1)$. Analogously 
\begin{eqnarray}\label{CL}
			F(\gamma(t)^{s_t}) - F(\gamma(0)^{s_0}) \leq  F(\gamma(t)^{s_0}) - F(\gamma(0)^{s_0})	=
			s_0^2 \int_{\R^2}\nabla \gamma(\tau_4 t) \cdot \nabla \gamma^{'}(\tau_4 t) t \, dx  
			 \nonumber \\ 
			+  \gamma \int_{\R^2}\int_{\R^2} log |x-y| (\gamma(\tau_5 t))^2(x) \gamma(\tau_2 t) (y) \gamma^{'}(\tau_5 t) (y) \, dx dy.
			\nonumber \\ -
			a s_0^{p-2} \int_{\R^2}|\gamma(\tau_6 t)|^{p-2} \gamma(\tau_6 t) \gamma^{'}(\tau_6 t) t \, dx  \nonumber
			\end{eqnarray}
for some $\tau_4, \tau_5, \tau_6 \in (0,1)$.  Now from (\ref{limite}) we deduce that
\begin{eqnarray}\label{CL}
			\lim_{t \to 0}\frac{I^{+}(\gamma(t)) - I^{+}(\gamma(0))}{t} = 	s_u^2 \int_{\R^2}\nabla u \cdot \nabla \psi \, dx  +  \gamma \int_{\R^2}\int_{\R^2} \log |x-y| u^2(x) u(y) \psi(y) \, dx dy
			\nonumber \\
		 -
			a s_u^{p-2} \int_{\R^2}|u(x)|^{p-2} u(x) \psi(x) \, dx \nonumber \\
			= \int_{\R^2} \nabla(u^{s_u}) \cdot \nabla(\psi^{s_u}) \, dx +  \gamma \int_{\R^2}\int_{\R^2} \log |x-y| (u^{s_u})^2(x) u^{s_u}(y) \psi^{s_u}(y) \, dx dy \nonumber \\
			+ \gamma \log(s_u)  \int_{\R^2}\int_{\R^2} u^2(x) u(y) \psi(y) \, dx dy - a \int_{\R^2}|u^{s_u}(x)|^{p-2} u^{s_u}(x) \psi^{s_u}(x) \, dx. \nonumber \\
			= DF(u^{s_u})[\psi^{s_u}] + \gamma \log (s_u) \int_{\R^2} u^2(x) \, dx \int_{\R^2}   u(y) \psi(y) \, dy = DF(u^{s_u})[\psi^{s_u}] \nonumber
			\end{eqnarray}
			for every $u \in S(c)$ and $\psi \in T_u S(c)$. 
\end{proof}

In our next lemma $I^{\pm}$ denotes either $I^+$ or $I^-$ and accordingly $\Lambda^{\pm}(c)$ denotes $\Lambda^+(c)$ (or $\Lambda^-(c)$) and $s_u = s_{u}^+$ (or $s_u = s_{u}^-$).
			
\begin{lem}
\label{ps}
Let $\mathcal{G}$ be a homotopy stable family of compact subsets of $S(c)$ with closed boundary $B$ and let
$$
	\textcolor{blue}{e_{\mathcal{G}}^\pm}:= \inf_{A\in \mathcal{G}}\max_{u\in A}I^{\pm}(u).$$ 
Suppose that $B$ is contained in a connected component of $\Lambda^{\pm}(c)$ and that 
$\max\{\sup I^{\pm}(B),0\}<
\textcolor{blue}{e_{\mathcal{G}}^\pm}
<\infty$. Then there exists a Palais-Smale sequence $(u_n) \subset \Lambda^{\pm}(c)$ for $F$ restricted to $S(c)$ at level $\textcolor{blue}{e_{\mathcal{G}}^\pm}$.
\end{lem}

\begin{proof}
Take $(D_n) \subset \mathcal{G}$ such that $\max_{u\in D_n}I^{\pm}(u)<
\textcolor{blue}{e_{\mathcal{G}}^\pm}
+\frac1n$ and
$$\eta :[0,1]\times S(c)\rightarrow S(c),\ \eta (t,u)=u^{{1-t + ts_u}}.$$
Since $s_u =1$ for any $u\in \Lambda^{\pm}(c)$, and $B\subset \Lambda^{\pm}(c)$, we have $\eta (t,u)=u$ for $(t,u)\in (\{0 \}\times S(c))\cup ([0,1]\times B)$. Observe also that $\eta$ is continuous. Then, using the definition of $\mathcal{G}$, we have
$$A_n:= \eta (\{1 \}\times D_n )=\{u^{s_u}:\ u\in D_n \}\in \mathcal{G}.$$
Also notice that $A_n \subset \Lambda^{\pm}(c)$ for all $n \in \N$. Let $v\in A_n$, i.e. $v=u^{s_u}$ for some $u\in D_n$ and $I^{\pm}(u)=I^{\pm}(v)$. So $\max_{A_n}I^{\pm}=\max_{D_n}I^{\pm}$ and therefore $(A_n) \subset \Lambda^{\pm}(c)$ is another minimizing sequence of 
$
\textcolor{blue}{e_{\mathcal{G}}^\pm}$.
Using the  minimax principle \cite[Theorem 3.2]{Gh}, we obtain a Palais-Smale sequence $(\tilde{u}_n)$ for $I^{\pm}$ on $S(c)$ at level 
$\textcolor{blue}{e_{\mathcal{G}}^\pm}$ such that
 $dist_{X} (\tilde{u}_n , A_n)\rightarrow 0$ as $n\rightarrow \infty$. 
 Now  writing $s_n=s_{\tilde{u}_n}$ to shorten the notations, we set $u_n=\tilde{u}_n^{s_n}\in \Lambda^{\pm}(c)$.  We claim that there exists $C>0$ such that,
 \begin{equation}
\label{BaSoe1}
\frac1C \leq s_n^2\leq C
\end{equation}
for $n \in \N$ large enough. 
Indeed, notice first that
 \begin{equation}
\label{BaSoe2}
s_n^2=\dfrac{A(u_n)}{A(\tilde{u}_n)}.
\end{equation}
Since by definition we have $F(u_n)=I^{\pm}(\tilde{u}_n )\rightarrow 
\textcolor{blue}{e_{\mathcal{G}}^\pm}
$, we deduce from Lemma \ref{coercive}, that there exists $M>0$ such that 
 \begin{equation}
\label{BaSoe3}
\frac1M\leq A(u_n) \leq M.
\end{equation}
On the other hand, since $(A_n) \subset \Lambda^{\pm}(c)$, is a minimizing sequence for $
\textcolor{blue}{e_{\mathcal{G}}^\pm}$
and $F$ is $H$ coercive on $\Lambda^{\pm}(c)$, we deduce that $(A_n)$ is uniformly bounded in $H$ and thus from $dist_X(\tilde{u}_n , A_n)\rightarrow 0$ as $n\rightarrow \infty$, it implies that $\sup_n A(\tilde{u}_n) <\infty$. Also, since $A_n$ is compact for every $n \in \N$, there exists a $v_n \in A_n$ such that  $dist_X(\tilde{u}_n , A_n)=\|v_n-\tilde{u}_n \|_X$  and, using once again Lemma \ref{coercive}  we also deduce that, for a $\delta >0$,
$$A( \tilde{u}_n)\geq A(v_n)- A(\tilde{u}_n - v_n)\geq 
\dfrac{\delta}{2}.$$
This proves the claim.

Next, we show that $(u_n) \subset \Lambda^{\pm}(c)$ is a Palais-Smale sequence for $F$ on $S(c)$ at level 
$\textcolor{blue}{e_{\mathcal{G}}^\pm}$.
Denoting by $\|.\|_\ast$ the dual norm of $(T_{u_n}S(c))^\ast$, we have
$$
\|dF(u_n)\|_\ast=\sup_{\psi \in T_{u_n}S(c),\ \|\psi\| \leq 1}|dF(u_n)[\psi]|=\sup_{\psi \in T_{u_n}S(c),\ \|\psi\| \leq 1}|dF(u_n)[ (\psi^{-s_n})^{s_n}]|.
$$
From Lemma \ref{isomorphism} we know that  $T_{\tilde{u}_n} S(c)\rightarrow T_{u_n} S(c)$ defined by $\psi \rightarrow \psi^{s_n}$ is an isomorphism. 
Also, from Lemma \ref{transform} we have that $dI^{\pm}(\tilde{u}_n)[\psi^{-s_n}]=dF(u_n)[(\psi^{-s_n})^{s_n}]$. It follows that
\begin{equation}\label{liencleff}
\|dF(u_n)\|_\ast = \sup_{\psi \in T_{u_n}S(c),\ \|\psi\| \leq 1} |dI^{\pm}(\tilde{u}_n)[\psi^{-s_n}]|.
\end{equation}
At this point it is easily seen from \eqref{BaSoe1} that (increasing $C$ if necessary)
$ \|\psi^{-s_n}\| \leq C \|\psi\| \leq C$
and we deduce from \eqref{liencleff} that $(u_n) \subset \Lambda^{\pm}(c)$ is a Palais-Smale sequence for $F$ on $S(c)$ at level 
$\textcolor{blue}{e_{\mathcal{G}}^\pm}$.
\end{proof}

\begin{lem}\label{psbisl}
There exists a Palais-Smale sequence  $(u_n) \subset \Lambda^+(c)$ for $F$ restricted to $S(c)$ at the level $\gamma^{+}(c)$ and a Palais-Smale $(u_n) \subset \Lambda^-(c)$ for $F$ restricted to $S(c)$ at the level $\gamma^{-}(c)$.
\end{lem}

\begin{proof}
Let us assume that $(u_n) \subset \Lambda^+(c)$, the other case can be treated similarly. We use Lemma \ref{ps} taking the set  $\bar{\mathcal{G}}$ of all singletons belonging to $S(c)$ and $B=\varnothing$. It is clearly a homotopy stable family of compact subsets of $S(c)$ (without boundary).  Since
$$e_{\bar{\mathcal{G}}}^{+}:=\inf_{A\in \bar{\mathcal{G}}}\max_{u\in A}I^{+}(u)=\inf_{u\in S(c)}I^{+}(u)=\gamma^+(c)$$
 the lemma follows directly from Lemma \ref{ps}.
\end{proof}

Now we are ready to give
		
\begin{proof}[Proof of Theorem \ref{infreachxxxx}]
		We give the proof for $u^+$, the one for $u^-$ is almost identical. Let $(u_n) \subset \Lambda^+(c)$ be a Palais-Smale sequence for $F$ restricted to $S(c)$ at level $\gamma^+(c)$ whose existence is insured by Lemma \ref{psbisl}. By Lemma  \ref{coercive} we know that  $(u_n)$ is bounded in $H$. Also since the functional $F$ is translational invariant, in view of Lemma \ref{X-bound} it is not restrictive to assume that $(u_n) \subset \Lambda^+(c)$ is bounded in $X$. At this point we conclude using Lemma \ref{convergence}. 
		\end{proof}

\subsection{Proof of Theorem \ref{double infinity}.}\label{subsection34}
		
We are now interested in the existence of infinitely many solutions lying on $\Lambda^+(c)$ and $\Lambda^-(c)$.  For this we shall work in the subspace $X_{rad}$ of $X$ consisting of radially symmetric functions. We set $\Lambda_{rad}(c) = \Lambda(c) \cap X_{rad}$. \medskip

We denote by $\sigma : X \rightarrow X$ the transformation $\sigma (u)=-u$. The following definition is \cite[Definition 7.1]{Gh}.
\begin{definition}
Let $B$ be a closed subset of a metric space $Y$. We say that a class $\mathcal{G}$ of compact subsets of $Y$ is a $\sigma$-homotopy stable family with closed boundary $B$ if
\begin{enumerate}
\item every set in $\mathcal{G}$ is $\sigma$-invariant.
\item every set in $\mathcal{G}$ contains $B$;
\item for any $A\in \mathcal{G}$ and any $\eta \in C([0,1]\times Y, Y)$ satisfying, for all $t\in [0,1]$, $\eta (t,u)=\eta (t,\sigma (u))$,  $\eta (t,x)=x$ for all $(t,x)\in (\{0\}\times Y)\cup ([0,1]\times B)$, we have $\eta (\{1\} \times A)\in \mathcal{G}$.
\end{enumerate}
\end{definition}

\begin{lem}
\label{psbis2}
Let $\mathcal{F}$ be a $\sigma$-homotopy stable family of compact subsets of $\Lambda_{rad}^{\pm}(c)$ with a close boundary $B$. Let $c_{\mathcal{F}}:= \inf_{A\in \mathcal{F}}\max_{u\in A}F(u)$.  Suppose that $B$ is contained in a connected component of $\Lambda_{rad}^{\pm}(c)$ and that $\max \{\sup F(B),0\}<c_{\mathcal{F}}<\infty$. Then there exists a Palais-Smale sequence $(u_n) \subset \Lambda_{rad}^{\pm}(c)$ for $F$ restricted to $S(c)$ at level $c_{\mathcal{F}}$.
\end{lem}

\begin{proof}
We are only sketchy here and refer to \cite{BaSo2} for the proofs of closely related results. The proof of Lemma \ref{psbis2} first relies  on an  equivariant version of Lemma \ref{ps}, whose proof is almost identical to the one of Lemma \ref{ps}. Then the lemma follows just as \cite[Theorem 3.2]{BaSo2} follows from \cite[Proposition 3.9]{BaSo2}.
\end{proof}

\begin{remark}
Lemma \ref{psbis2} establishes that, if the assumptions of the equivariant minimax principle \cite[Theorem 7.2]{Gh} are satisfied by the functional $F$ constrained to $\Lambda^{\pm}(c)$, then we can find a ``free" Palais-Smale sequence for $F$ on $S(c)$ made of elements of $\Lambda^{\pm}(c)$.
\end{remark}

Now let $\mathcal{H}:= \Lambda^{+}(c) \cap X_{rad}$ (or $\mathcal{H}:= \Lambda^{-}(c) \cap X_{rad}$) and recall, in this notation, the definition of the genus of a set due to M.A. Krasnosel'skii.

\begin{definition} \label{genus}
Let $\mathcal{A}$ be a family of sets $A \subset \mathcal{H}$ such that $A$ is closed and symmetric ($u \in A$ if and only if $-u \in A$). For every $A \in \mathcal{A}$, the genus of $A$ is defined by
$$
\gamma(A):= \min \{n \in \N: \exists \ \varphi : A \rightarrow \R^n\backslash \{0\}, \varphi \,\, \text{is continuous and odd}\}.
$$
When there is no $\varphi$ as described above, we set $\gamma(A)= \infty.$
\end{definition}
Let  $\mathcal{A}_\mathcal{H}$ be the family of compact and symmetric sets $A \subset \mathcal{H}$. For any $k \in \N^+$, define
$$
\Gamma_{k}:=\{A \in \mathcal{A}_\mathcal{H}: \gamma(A) \geq k\}
$$
and
$$
\beta_k:=\inf_{A \in \Gamma_{k}}\sup_{u \in A}F (u).
$$

\begin{lem} \label{nonempty} 
Let $c <c_0$. For  any $k \in \N^+$, $\Gamma_k^- \neq \emptyset$ and $\Gamma_k^+ \neq \emptyset$.
\end{lem}

\begin{proof}
We give the proof for $\Gamma_k^+$.
Let $V \subset X_{rad}$ be such that $\dim V=k$. We set $SV(c):=V \cap S(c)$. By the basic property of the genus, see \cite[Theorem 10.5]{AmMa}, we have that $\gamma(SV(c))= \dim V =k$. In view of Lemma \ref{intersection}, for any $u\in SV(c)$ there exists unique $s_u^+ >0$ such that $ u^{s_u^+} \in \Lambda^+(c)$. It is easy to check that the mapping $\varphi: SV(c)\rightarrow \Lambda(c)$ defined by $\varphi (u)= u^{s_u^+}$ is continuous and odd. Then \cite[Lemma 10.4]{AmMa} leads to $\gamma( \varphi(SV(c))) \geq \gamma(SV(c))=k$ and this shows that $\Gamma_k \neq \emptyset$. 
\end{proof}

\begin{proof}[Proof of Theorem \ref{double infinity}]
We give the proof for $\Lambda^+(c)$, the case of $\Lambda^-(c)$ is identical.
Consider the minimax level $\beta_k$. From Lemma \ref{nonempty} we know that each of the classes $\Gamma_k$ is non empty and thus to each of them we can apply Lemma \ref{psbis2} to obtain the existence of  Palais-Smale sequences $(u_n^k) \subset \Lambda^+_{rad}(c)$ for $F$ restricted to $S(c)$ at the levels $\beta_k$. Since $u_n^k$ is radial we know from Lemmas \ref{lem_V1(u_n)_bornee} and Lemma \ref{X-bound} that $(u_n^k) \subset X_{rad}$ is bounded in $X$. At this point we conclude using Lemma \ref{convergence} that $(u_n^k)_n$ converges to a $u^k$ which is a critical point of $F$ on $S(c)$.  Now to show that if two (or more) values of $\beta_k$ coincide, than $F$ has infinitely many critical points at level $c_k$, one can either proceed in the usual way, or adapt \cite[Lemma 6.4]{BaSo2} to the present setting.
\end{proof}

\section{The case $\gamma<0$} \label{gammanegatif}

In this section, for convenience, we change $\gamma$ into $-\gamma$ and thus we write
$$F(u)=\frac{1}{2}A(u)-\frac{\gamma}{4}V(u)-\frac{a}{p}\; C(u) $$
with $\gamma >0$. With this change note that the function  $g_u : (0, \infty) \to X$ becomes
$$g_u(t)= F(u^t)= \frac{t^2}{2} A(u) - \frac{\gamma}{4} V(u) + \frac{\gamma c^2}{4} \log t 
-\frac{a }{p}	\ t^{p-2} C(u).$$ 
Obviously we still have that $g_u$ is $C^2$ on $(0,  \infty)$ and
\begin{equation}\label{deriprima}
g'_u(t)=  \frac{1}{t} \bigl( t^2 A(u) + \frac{\gamma c^2}{4} 
-\frac{a (p-2)}{p} t^{p-2} C(u) \bigr).
\end{equation}

Firstly, we notice that if $a \leq 0$ and $p>2$,  for each $u \in S(c)$ the  {\sl fiber map}
$g_u(t):= F(u^t)$ is strictly increasing and so we can immediately derive Theorem\ref{nonexistence}.

Also note that
$$
\Lambda(c)= \{u \in S(c) \  | \ Q(u) =0\} = \{u \in S(c) \ | \ g_u'(1)=0\}.
$$
For future reference  observe that defining 
\begin{equation}\label{tstar}
{t^*_u} = \Bigl[\frac{a(p-2)^2 C(u)}{2p A(u)} \Bigr]^{1/(4-p)}
\end{equation}  we have  that
\begin{equation}\label{characterization-star}
2 A(u^{{t^*_u}}) = \frac{a(p-2)^2}{p} C(u^{{t^*_u}}).
\end{equation}
Furthermore notice that 
\begin{equation}\label{primaplus3}
2 A(u^t) < \frac{a(p-2)^2}{p} C(u^t), \quad \forall \, 0 < t < {{t^*_u}}
\end{equation}	
and 
\begin{equation}\label{secondaplus3}
2 A(u^t) > \frac{a(p-2)^2}{p} C(u^t), \quad \forall \, t > {{t^*_u}}.
\end{equation}

In what follows we always assume that $a >0$ and $p<4$. The following quantities will play a crucial role in this section,
$$
K_1=  \frac{1}{2^{\frac{4-p}{2}}} \frac{1}{K_{GN}} \frac{p}{2^{3-p} (p-2)^{\frac{p}{2}}(4-p)^{\frac{4-p}{2}}}.
$$
and 
$$K_2=2^\frac{4-p}{2}\; K_1=\frac{1}{K_{GN}}\; \frac{p }{2^{3-p}(p-2)^{\frac{p}{2}}(4-p)^\frac{4-p}{2}}.$$

\subsection{Properties of $\Lambda(c)$}\label{propertiesLambda}

\begin{lem}\label{fuori}
Assume that $\gamma <0$  and $p <4$. Then 
$$\Lambda(c) \neq \emptyset \quad \mbox{if and only if} \quad a \geq K_1 \gamma^{\frac{4-p}{2}} c^{3-p}.$$
\end{lem}
\begin{proof}
	Let $u \in S(c)$ and $t >0$. Defining 
	$$
	\phi_u(t)= Q(u^t) = A(u) t^2 - a \frac{(p-2)}{p} C(u) t^{p-2} + \frac{\gamma c^2}{4}
	$$
	we have $g'_u(t)= \frac{1}{t} \phi_u(t)$.
	Thus the function $\phi_u$ achieves its minimum at $t^*_u$ given in   \eqref{tstar}
	and 
	\begin{equation}\label{nocross}
	\phi_u({{t_u}^*})= 
	\gamma \frac{c^2}{4} - \tilde K_1 \Big[
	\frac{C(u)^2}{A(u)^{p-2}} \Big]^{1/(4-p)} a^{2/(4-p)}
	\end{equation}
	where
	$\displaystyle \tilde K_1= \frac{(p-2)^{p/(4-p)}(4-p)}{p^{2/(4-p)} 2^{(p-2)/(4-p)}}$. 
	By Gagliardo-Nirenberg inequality (\ref{GN}),  we have
	$$
	\frac{C(u)^2}{A(u)^{p-2}} \leq K^2_{G N} c^2
	$$
	which leads to 
	$$\inf_{u \in S(c)} Q(u)  \geq  \gamma \frac{c^2}{4} - \tilde K_1 [K_{G N} \ a \ c]^{2/(4-p)}.
	$$
	Hence if $a < K_1 \gamma^{\frac{4-p}{2}} c^{3-p}$, then
	$\inf_{S(c)} Q(u) >0$, and so $\Lambda(c)= \emptyset$. Now, since the best constant in the Gagliardo-Nirenberg inequality is reached, say by $\bar{u} \in S(c)$, we also have that
	$$\inf_{u \in S(c)} Q(u) = Q(\bar{u}) =   \gamma \frac{c^2}{4} - \tilde K_1 [K_{G N} \ a \ c]^{2/(4-p)}.$$
	Thus if  $a > K_1 \gamma^{\frac{4-p}{2}} c^{3-p}$, then
	$$\inf_{u \in S(c)} Q(u) <0 $$ and since $\lim_{t \to  \infty} \phi_u(t) = + \infty$, we deduce by continuity that $\Lambda(c)\neq \emptyset$.
	If 
	$a= K_1 \gamma^{\frac{4-p}{2}} c^{3-p}$, then $Q(\bar{u})=0$ and so $\Lambda(c) \neq \emptyset$. 
\end{proof}

\begin{lem}\label{unboundedbelow}
Assume that $\gamma <0$ and $p <4$. Then if
\begin{equation}\label{star}
 a >  K_1\; \gamma^\frac{4-p}{2}\; c^{3-p},
 \end{equation}
we have that $\underset{\Lambda(c)}{\inf} \;F = - \infty.$
\end{lem}

\begin{proof}
Our proof borrows ideas from \cite{BeJeLu}.
First observe that if $u \in S(c)$ is such that $Q(u) \leq 0$ then since $Q(u^t) \to + \infty$ as $t \to  \infty$ there exists a $t\geq 1$ such that $Q(u^t) =0$ and $F(u^t) \leq F(u)$. So we only need to prove that there exists a sequence $(u_n) \subset S(c)$ with $Q(u_n) \leq 0$ and $F(u_n) \to - \infty$ as $n \to \infty$. 
Let $c>0$ satisfies (\ref{star}) and assume first that $p >3$. Then there exists a $c_1 >0$ such that $c>c_1$ and $\{ u \in S(d): Q(u) <0\} \neq \emptyset$ for $d >c_1$. We set $\eta = c -c_1 >0$ and take $u \in C_0^{\infty}(\R^2),$ $u \geq 0$ with $||u||_2^2 = c - \frac{\eta}{2}$ and $Q(u) <0$. We also choose a $v \in C_0^{\infty}(\R^2),$ $v \geq 0$ with $||v||_2^2 =  \frac{\eta}{2}$. We now consider the  sequence 
$$u_n(x) = u(x) + \frac{1}{n}v\Big(\frac{1}{n}(x-nR)\Big) := u(x) + v_n(x)$$
where $R>0$ is choosen sufficiently large so that the supports of $u$ and $v_n$ are disjoints. Clearly
\begin{eqnarray*}
Q(u_n)  & = & A(u+v_n) - a \frac{p-2}{p}C(u+v_n) + \frac{\gamma c^2}{4} \\
& = &  A(u) - a \frac{p-2}{p}C(u) + \frac{\gamma c^2}{4} + A(v_n) - a \frac{p-2}{p}C(v_n)\\
& \rightarrow & A(u) - a \frac{p-2}{p}C(u) + \frac{\gamma c^2}{4} < 0,
 \end{eqnarray*}
since $A(v_n) \to 0$ and $C(v_n) \to 0$ as $n \to \infty$. Also we easily observe that, because the functions $u$ and $v_n$ are non negative, that $V_1(u_n) \geq V_1(v_n)$ and that $V_1(v_n) \to + \infty$ as $n \to  \infty$. We then deduce that $F(u_n) \to - \infty$ proving the lemma in the case $p >3$. 

Now if we assume that $p \leq 3$ there exists a $c_1 >0$ ($c_1 = + \infty$ if $p=3$) such that $c < c_1$ and $\{ u \in S(d): Q(u) <0\} \neq \emptyset$ for $d <c_1$ We then modify the previous proof by taking $u \in C_0^{\infty}(\R^2)$, $u \geq 0$ with $||u||_2^2 = \frac{c}{2}$ and $Q(u) <0$ and consider instead the sequence
$$u_n(x) = u(x) + \frac{1}{n}u\big(\frac{1}{n}(x- nR)\big).$$
By similar arguments we obtain  $Q(u_n) \to Q(u) <0$ and $ F(u_n) \to - \infty$ as $n \to \infty$.
\end{proof}

\begin{lem}\label{coercive2} 
Assume that $\gamma <0$, $a >0$ and $p <4$. Then,
\begin{enumerate}
\item $F$ restricted to $\Lambda(c) $ is bounded from above.
\item For any $m_1 \in \R$, there exists a $m_2 \in \R$ such that, for all $u \in \Lambda(c)$, $A(u) \leq m_2$  and $V_1(u) \leq m_2$ if $F(u) \geq m_1$.
\end{enumerate}
\end{lem}
\begin{proof}
Let $u \in \Lambda(c)$.
From 
\begin{equation}\label{3}
	C(u) = \frac{p}{a(p-2)}\Big[ A(u) + \frac{\gamma c^2}{4}\Big]
	\end{equation}
and  $\gamma V_1(u) \geq 0$, we deduce
	\begin{equation*}
	F(u) \leq \frac{1}{2} A(u) + \frac{\gamma c^{3/2}}{4} A(u)^{\frac{1}{2}} - \frac{1}{p-2}\Big[A(u) +\frac{\gamma c^2}{4}\Big] \leq - \frac{(4-p)}{2(p-2)}  A(u) + \frac{\gamma c^{3/2}}{4} A(u)^{\frac{1}{2}}.
	\end{equation*}
	Since $2<p<4$, both points follow.
\end{proof}

The following three lemmas give information on the geometric structure of $\Lambda(c)$.
\begin{lem}\label{ilu}
\label{lem_Q=0_A(u)=k_0}
Assume that $\gamma <0$, $a >0$ and $p <4$. 
If $Q(u)\leq 0$ (resp. $Q(u) <0)$  and $A(u)=k_0$ then
$$ a\geq K_2\; \gamma^\frac{4-p}{2}\; c^{3-p} \, \Big(\mbox{resp. }  a > K_2\; \gamma^\frac{4-p}{2}\; c^{3-p}\Big)$$
\end{lem}

\begin{proof}
Since $Q(u)\leq 0$, we have
$$ A(u)\leq a \frac{p-2}{p}C(u)-\frac{\gamma c^2}{4}.$$
Then,  by Gagliardo-Nirenberg and since $A(u)=k_0$, we get 
\begin{align*}
\frac{p-2}{4-p}\;\frac{\gamma c^2}{4}& \leq
 a \frac{p-2}{p} K_{GN}\left[\frac{p-2}{4-p\;}\frac{\gamma c^2}{4}\right]^{\frac{p-2}{2}}\;c-\frac{\gamma c^2}{4}\\
\left[\frac{p-2}{4-p}+1 \right]\;\frac{\gamma c^2}{4}= \frac{1}{4-p}\frac{\gamma c^2}{2}& \leq
 a  K_{GN}\frac{(p-2)^\frac{p}{2}}{p (4-p)^\frac{p-2}{2}2^{p-2}}\gamma^\frac{p-2}{2}c^{p-1}\\
 \gamma^\frac{4-p}{2}c^{3-p}\; \frac{1}{K_{GN}}\frac{p}{(4-p)^\frac{4-p}{2} 2^{3-p}(p-2)^\frac{p}{2}}
 &\leq a\\
  K_2\; \gamma^\frac{4-p}{2}\; c^{3-p}&\leq a,
 \end{align*}
 whence the result.
\end{proof}

\begin{lem}
\label{lem_A(u)=k_0_2}
Assume that $\gamma <0$, $a >0$ and $p <4$. 
Let $u\in S(c)$ such that $Q(u)=0$ and $\left.\frac{d}{dt}\right|_{t=1}\;Q(u^t)=0$.
Then $ A(u)=k_0.$
\end{lem}
\begin{proof}
First, a simple computation shows that 
$$\left.\frac{d}{dt}\right|_{t=1}\;Q(u^t)=2A(u)-a\frac{(p-2)^2}{p}\; C(u).$$
So by hypothesis, $$a\frac{p-2}{p}\;C(u)=\frac{2}{p-2}\; A(u).$$
But we also know that $Q(u)=A(u)-a\;\frac{p-2}{p}\; C(u)+\gamma\;\frac{c^ 2}{4}=0$, so
$$\left(\frac{2}{p-2}-1\right)\; A(u)=\gamma\;\frac{c^ 2}{4},$$ i.e.
$A(u)=k_0.$
\end{proof}

\subsection{Proof of Theorem \ref{lem_T(c)sousvariete234} }\label{Onesolution}

In order to prove Theorem \ref{lem_T(c)sousvariete234} we first establish the following lemma.

\begin{lem}
\label{lem_maximum_local}
Assume that $\gamma <0$ and $p<4$. Then if
 $K_1\; \gamma^\frac{4-p}{2}\; c^{3-p} \leq a \leq K_2\; \gamma^\frac{4-p}{2}\; c^{3-p}, $
 $M : = \underset{\Lambda(c)}{\sup} \;F $ is achieved on $\Lambda(c)$. 
\end{lem}

\begin{proof}
From Lemma \ref{coercive2} we already know that $M<\infty$ and that for any maximizing sequence $(u_n)  \subset \Lambda(c)$, $(A(u_n))$ is bounded. Clearly also $(V_1(u_n))$ is bounded.
Hence, using previous arguments  we may assume  that, up to a subsequence and translations, $(u_n)$ is bounded in $X$ and that, for some $u \in S(c)$,
$u_n\overset{X}{\rightharpoonup} u\in S(c)$ 
and 
$u_n\overset{H}{\rightharpoonup} u.$ In addition we have that  $V_2(u_n)\rightarrow V_2(u)$. At this point it is convenient to introduce the functional 
\begin{equation}\label{restriction}
G(u)=- \frac{4-p}{2(p-2)}\; A(u)-\frac{\gamma}{4}V(u)-\gamma\; \frac{c^2}{4(p-2)}.
\end{equation}
which coincide with $F$  on the set $\Lambda(c)$. Since $A$ (resp. $V_1$) is lowersemicontinuous for the weak convergence on $H$ (resp. X)
and since $G$ is invariant by translation, we deduce that $M\leq G(u).$

Similarly, $C$ is continuous for the weak convergence in $X$ and $A$ is lower semicontinuous for the weak convergence in $H$, hence $Q(u)\leq 0$. To conclude we just need to show that $Q(u)=0$. Observe that by a direct calculation, for any $t>0$,
$$G(u^t)=- \frac{4-p}{2(p-2)}\; A(u)\; t^2+\frac{\gamma c^2}{4}\log t-\frac{\gamma}{4}V(u),$$
and thus
\begin{equation}\label{derivee}
\frac{d}{dt}G(u^t)=- \frac{4-p}{(p-2)}\; A(u)\; t +\frac{\gamma c^2}{4}\frac{1}{t}.
\end{equation}
Note that $ v:= \left. \frac{d}{dt}\right|_{t=1}G(u^t)= 0$ is equivalent to $A(u)=k_0$. Since $Q(u) \leq 0$ we thus know from Lemma \ref{lem_Q=0_A(u)=k_0}  that  
$v \neq 0$.
We shall now prove that neither $v < 0$ nor $v  > 0$ is possible if $Q(u)<0$ and it will end the proof. 
	
First assume that $v < 0$. Since $Q(u^t)
\to  \displaystyle{\frac{\gamma c^2}{4} >0}$ as $t \to 0$, assuming that $Q(u) <0$, there exists a $t_0 <1$ such that $Q(u^{t_0})=0$ and $Q(u^t) \leq 0$ if $t \in [t_0,1]$. Thus, again by Lemma \ref{lem_Q=0_A(u)=k_0}, we deduce that $\frac{d}{dt}G(u^t) <0$ for $t \in [t_0,1]$ and consequently  $G(u^{t_0}) = F(u^{t_0}) > M$ in contradiction with the definition of $M$. Assume now that $ v > 0$. Since $Q(u^t)
\to  + \infty$ as $t \to \infty$ there exists a $t_1 >1$ such that $Q(u^{t_1})=0$ and $Q(u^t) \leq 0$ if $t \in [1,t_1]$. Thus, again by Lemma \ref{lem_Q=0_A(u)=k_0}, we deduce that $\frac{d}{dt}G(u^t) >0$ for $t \in [1,t_1]$ which lead to the same contradiction.
\end{proof}

At this point we are ready to give

\begin{thm}\label{lem_T(c)sousvariete234}
	Assume $\gamma <0$, $p <4$ and 		
	$$K_1\; \gamma^\frac{4-p}{2}\; c^{3-p}\leq a< K_2\; \gamma^\frac{4-p}{2}\; c^{3-p}.$$
	Then $\sup_{\Lambda(c)} F(u) < \infty$  and it is achieved by a critical point of $F$ restricted to $S(c)$.
\end{thm}
\begin{proof}
We shall see in Lemma \ref{l2} that, under the assumptions of the theorem, 
 $\Lambda(c)$ is a submanifold of X of codimension 2.
By Lemma \ref{lem_maximum_local} we  know that there exists $u\in \Lambda(c)$ such that $$F(u)=\underset{\Lambda(c)}{\max} \;F.$$
Since $u$ is a maximizer of $F$ on $\Lambda(c)$, hence a critical point, 
there exist two Lagrange multipliers $\lambda,\mu$ such that
\begin{equation}\label{t1}
dF(u)= \lambda \;(u,\cdot)+\mu\; dQ(u).
\end{equation}
Our aim is to show that $\mu =0$. Observe that (\ref{t1}) can be rewritten as
\begin{equation}\label{t2}
- (1 - 2 \mu) \Delta u - \gamma (log|\cdot|*|u|^2)u - \lambda u = a (1 - \mu(p-2)) |u|^{p-2}u
\end{equation}
and thus from Lemma \ref{lem_Q=0} we obtain that
\begin{equation}\label{t3}
(1- 2 \mu) A(u) - a \frac{(p-2)(1- \mu(p-2))}{p} C(u) + \frac{\gamma c^2}{4}=0.
\end{equation}
Now using that $Q(u)=0$ we obtain from (\ref{t3}) that
\begin{equation}\label{t4}
 \frac{p-2}{p}\mu (4-p) C(u) = \frac{\gamma c^2}{2}\mu.
\end{equation}
If $\mu =0$ we are done, so we assume that $\mu \neq 0$. We then deduce that
\begin{equation}\label{t5}
C(u) = \frac{p}{a(p-2)(4-p)} \frac{\gamma c^2}{2}
\end{equation}
and inserting (\ref{t5}) into $Q(u)=0$ we deduce that $A(u) = k_0$. This contradiction proves that $\mu=0$ namely that $u$ is a critical point of $F$ restricted to $S(c)$.
\end{proof}

\begin{remark} We also would like to express the sufficient conditions given in Theorem \ref{lem_T(c)sousvariete234}  in term of $c >0$ since an  interesting phenomenon then occurs. 
Actually, there is a strong qualitative change depending on the position of $p$ with respect to the, thus critical, exponent $3$. 
In particular, our result says that $F|_{S(c)}$ has no critical point 
\begin{itemize}
\item if $c >0$ is large for $2<p<3$.
\item if $c>0$ is small for $3<p<4$.
\item if $a< K_1\; \gamma^\frac{4-p}{2}$ but without condition on $c >0$ if $p=3$.
\end{itemize}
In the following table, we express the sufficient conditions given by Theorem \ref{lem_T(c)sousvariete234} in term of $c$.
\medskip

\begin{tabular}{|c|c|c|c|}
\hline 
  & Nonexistence & Existence &$c_i,\;i=1,2$\\ 
\hline 
$2<p<3$  & $c>c_1$ & $c_2<c\leq c_1$&$\frac{1}{{K_i}^{\frac{1}{3-p}}}\frac{a^\frac{1}{3-p}}{\gamma^\frac{4-p}{2(3-p)}}$ \\ 
\hline 
$p=3$ &   $a<  K_1\; \gamma^\frac{4-p}{2}$ &$K_1\; \gamma^\frac{4-p}{2}\leq a< K_2\; \gamma^\frac{4-p}{2}$& \qquad\qquad\qquad\qquad\qquad\qquad \\ 
\hline 
$3<p<4$    &$c<c_1$& $c_1\leq c< c_2$ &${K_i}^{\frac{1}{p-3}}\frac{\gamma^\frac{4-p}{2(p-3)}}{a^\frac{1}{p-3}}$\\ 
\hline 
\end{tabular} 
\end{remark}


\subsection{Proof of Theorem \ref{infreach233} }\label{Twosolution}

Considering a sequence $(u_n) \subset S(c)$ such that $C(u_n) =1$ and $A(u_n) \to \infty$, we deduce from (\ref{nocross}) that for any $a>0$ and $c >0$ there always exists a $u \in S(c)$ such that $u^s  \not \in \Lambda(c)$ for any $s>0$. For this reason we shall localized our search of critical points into the subset of $S(c)$ given by
$$		
V= \{u \in S(c) \  |  \ (t^*_u)^2 A(u) > k_0\}.
$$	
The following result gives an alternative characterization of $V$ and some first properties.
\begin{lem}\label{vuoto}
Assume that $\gamma <0$, $a >0$ and $p<4$. We have
\begin{enumerate}
\item $u \in V \Longleftrightarrow \inf_{t>0}Q(u^t) = Q(u^{t_u^*}) <0.$
\item 
	If $a > K_1 \gamma^{\frac{4-p}{2}} c^{3-p}$, then $V$ is an open, not empty subset in $S(c)$.
\end{enumerate}
\end{lem}
\begin{proof}
By definition, $u \in V$ if and only if
\begin{equation}\label{491}
A(u^{{t^*_u}}) > k_0 = \frac{(p-2)}{(4- p)} \frac{\gamma c^2}{4}.
\end{equation}
But \eqref{491} is equivalent to 
\begin{equation}\label{492}
A(u^{{t^*_u}}) + \frac{\gamma c^2}{4} < 
	\frac{2}{p-2} A(u^{{t_u^*}})
\end{equation}
and recording that by definition of $t_u^*$, 
$$ 
	\frac{2}{p-2} A(u^{{t^*_u}}) = a \frac{(p-2)}{p}C(u^{{t^*_u}})
	$$
it is also equivalent to 
$$ 
	A(u^{{t^*_u}})  a \frac{(p-2)}{p}C(u^
	{{t^*_u}})  + \frac{\gamma c^2}{4} < 0
	$$
namely to $Q(u^{{t^*_u}})<0.$ This proves the first point. Now, arguing as in the proof of 
Lemma \ref{fuori},  we see that if
	$a >  K_1 \gamma^{\frac{4-p}{2}} c^{3-p}$, there exists $u \in S(c)$ such that $Q(u) < 0$ proving that $V$ is non empty. The 
	fact that $V$ is open in $S(c)$, follows from the continuity of the map $u \mapsto t_u^*$.
\end{proof}

\begin{remark}\label{advantageV}
For future reference note that it can be checked, by direct calculations, that if $Q(u)=0$ and $A(u)=k_0$ then $t_u^*=1$ and thus $u \notin V$.
\end{remark}

Our next result can be deduced from the characterization of $V$ given in Lemma \ref{vuoto} but we provide here a proof directly based on the definition of $V$.

\begin{lem}\label{dentro22}
Assume that $\gamma <0$ and $p <4$.
	Let $a > K_1 \gamma^{\frac{4-p}{2}} c^{3-p}$, then for any $u \in V$, we have that 
	$ u^s \in V$ for any $s>0$.
\end{lem}
\begin{proof}
	Let $u \in V$, namely $u \in S(c)$ and $(t^*_u)^2 A(u) >k_0$. 
	We define $v=  u^s$ and
	we evaluate 
	$$ 
	{{t^*_v}} =  \Bigl[ a \frac{(p-2)^2}{2p}
	\frac{C(v)}{A(v)} \Bigr]^{1/(4-p)} 
	= \Bigl[ \frac{(s)^{p-2}}{(s)^2}
	a \frac{(p-2)^2}{2p} \frac{C(u)}{A(u)} \Bigr]^{1/(4-p)} = \frac{{t^*_u}}{s}.
	$$
	It follows that 
	$$
	{{t^*_v}}^2 A(v)= \frac{({t^*_u})^2}{s^2} s^2 A(u) =
	(t^*_u)^2 A(u)
	>k_0
	$$
	and thus $v \in V$.
\end{proof}

Let us now denote 
$$
\Lambda^+(c) = \{ u \in V \ |\ g'_u(1)=0, \ g''_u(1) >0 \}, 
$$
$$
\Lambda^0(c) = \{ u \in  V \ |\ g'_u(1)=0, \ g''_u(1) =0 \},
$$
$$
\Lambda^-(c) = \{ u \in  V \ |\ g'_u(1)=0, \ g''_u(1) <0 \}.
$$
Observe that  $\Lambda^0(c) = \emptyset$  by Lemma \ref{lem_A(u)=k_0_2} and Remark \ref{advantageV}.

\begin{lem}\label{minimo8}
Let $a > K_1 \gamma^{\frac{4-p}{2}} c^{3-p}$. For any $u \in V$, there exists 
	\begin{enumerate}
		\item a unique $s_u^+ >0$ such that  $ u^{s_u^+} \in  \Lambda^+(c)$.
		Such $s_u^+$ is a strict local minimum point for $g_u$.
		\item a unique $s_u^- >0$ such that  $ u^{s_u^-} \in  \Lambda^-(c)$.
		Such $s_u^-$ is a strict local maximum point for $g_u$.
	\end{enumerate} 
\end{lem}
\begin{proof}
	Fix $u \in V$. Since $({t^*_u})^2 A(u) > k_0$, we deduce that 
	
	\begin{align}
g_u'({{t^*_u}}) & = \frac{1}{{t^*_u}} 
	\bigl( A(u^{{t^*_u}}) + \frac{\gamma c^2}{4} - a \frac{(p-2)}{p} C(u^{{t^*_u}}) \bigr) \nonumber \\
& =  \frac{1}{{t^*_u}} \bigl( A(u^{{t^*_u}}) + \frac{\gamma c^2}{4} -  \frac{2}{(p-2)} A(u^{t^*_u}) \bigr)   = \frac{1}{{t^*_u}} \bigl( \frac{\gamma c^2}{4} - \frac{(4-p)}{(p-2)} {{t^*_u}}^2 A(u) \bigr) <0.
\label{keypoint}
\end{align}
	Moreover by $(\ref{primaplus3})$ we have that for any $t \in (0, {t^*_u})$ we have
	
	\begin{align}
g_u'(t) & = \frac{1}{t} 
	\bigl( A(u^{t}) + \frac{\gamma c^2}{4} - a \frac{(p-2)}{p} C(u^{t}) \bigr)  \nonumber \\
& <  \frac{1}{{t}} \bigl( A(u^{{t}}) + \frac{\gamma c^2}{4} -  \frac{2}{(p-2)} A(u^{t}) \bigr)   = \frac{1}{{t}} \bigl( \frac{\gamma c^2}{4} - \frac{(4-p)}{(p-2)} {{t}}^2 A(u) \bigr).
\label{key2}
\end{align}
	By $(\ref{key2})$ we infer that there exists $\delta >0$ such that  for any $t \in ({t^*_u} - \delta, {t^*_u})$,
	$g_u'(t) <0$ and thus $g_u(t)$ is decreasing in $({t^*_u} -\delta, {t^*_u})$.
	Taking into account that the function $g_u(t) \to - \infty $ as $t \to 0^+$ and $g_u(t) \to + \infty$ as $t \to + \infty$,
	we conclude that there exists at least a critical point $s^+_u> t_u^*$ which is a local minimum point of $g_u$ and a critical point 
$s_u^- < t_u^*$ which is a local maximum point of $g_u$.

	Since $s^+_u >t^*_u$, from  $(\ref{secondaplus3})$ we derive that
	\begin{equation}\label{disc}
	2 (s^+_u)^2 A(u)  - a \frac{(p-2)^2}{p} (s^+_u)^{p-2} C(u) >0.
	\end{equation}
	Moreover from $(\ref{disc})$ and the fact that $g_u'(s^+_u)=0$, we derive that 	 	
	\begin{eqnarray}\label{pos3}
	g_u''(s^+_u)=  \frac{1}{(s^+_u)^2} 
	\bigl(A(u^{s^+_u}) - \frac{\gamma c^2}{4} - a \frac{(p-2)(p-3)}{p} C(u^{s^+_u}) \bigr) \\ 
	= \frac{1}{(s^+_u)^2} 
	\bigl( 2 (s^+_u)^2 A(u)  - a \frac{(p-2)^2}{p} (s^+_u)^{p-2} C(u) \bigr) > 0
	\end{eqnarray}
	
	Therefore $s^+_u$ is a strict minimum point for $g_u$ 	
	and $u^{s^+_u} \in \Lambda^+(c)$.

	We have to show that $s^+_u$ is unique.
	By contradiction we assume that 
	there exists  $z^+_u>0$ an other critical point of $g_u$ which is a local minimum point.

	Firstly we observe that if $0< z^+_u <t_u^*$, then from  $g_u'(z^+_u)=0$ and $(\ref{primaplus3})$ it results 
	\begin{eqnarray}\label{nib2}
	g_u''(z^+_u)
	= \frac{1}{(z^+_u)^2} 
	\bigl( 2 (z^+_u)^2 A(u)  - a \frac{(p-2)^2}{p} (z^+_u)^{p-2} C(u) \bigr) < 0
	\end{eqnarray}
	which is a contradiction.
	This implies that $z^+_u> t_u^*$ and thus arguing as before we have 
	$g_u''(z^+_u) < 0$ namely  $u^{z^*_u} \in \Lambda^+(c)$. 
	We derive the existence of an other critical point $\theta_u >t_u^*$, which is a local maximum for $g_u$.
	Taking into account $(\ref{secondaplus3})$, we again deduce
	$g_u''(\theta_u) > 0$, which is a contradiction.
	Therefore the point $s^+_u$ is unique.

Now a direct adaptation of the argument used for $s^-_u$ leads to conclude that $s_u^+>0$ is the unique local maximum for $g_u$.
\end{proof}

For future reference note
\begin{lem}\label{regularite2}
Assume that $\gamma <0$ and $p<4$. If 
$a > K_1\; \gamma^\frac{4-p}{2}\; c^{3-p} $,	
the maps $u \in V \mapsto s_{u^+} \in \R$ and $u \in V \mapsto s_{u^-}\in \R$ are of class $C^1$.
\end{lem}
\begin{proof}
It is a direct application of the Implicit Function Theorem on the $C^1$ function
$\Psi : \R \times V  \to \R$, defined by  				
$\Psi(s,u)= g'_u(s)$, taking into account that  $\Psi(s^\pm_u,u)=0$, $\partial_s \Psi(s^+_u,u) = g_u''(s^+_u) >0$, 
$\partial_s \Psi(s^+_u,u) = g_u''(s^-_u) <0$ and $\Lambda^0(c)= \emptyset$ by  Lemma \ref{lem_A(u)=k_0_2} and Remark \ref{advantageV}.
\end{proof}

	\begin{lem}\label{lem_T(c)sousvariete2-Louis}
	Let $\gamma <0$ and $p <4$. Assume that  $a > K_1 \gamma^{\frac{4-p}{2}} c^{3-p}$, 
			 then $\Lambda(c) \cap V $ is a submanifold, of class $C^1$, of codimension $2$ of $X$ and a submanifold of codimension $1$ in $S(c)$.
		\end{lem}
		
\begin{proof}
Note that the assumption $a > K_1 \gamma^{\frac{4-p}{2}} c^{3-p}$ is just used to guarantee that $V$ is an open, not empty subset in $S(c)$.
By definition, $u \in \Lambda(c)$ if and only if $G(u):=\|u\|_2^2 -c =0$ and $Q(u)=0.$ It is easy to check that $G, Q$ are of $C^1$ class. Hence we only have to prove that for any $u \in \Lambda(c)$,
$$
(dG(u), dQ(u)): X \rightarrow \R^2 \,\, \mbox{is surjective}.
$$
If this failed, we would have that $dG(u)$ and $dQ(u)$ are linearly dependent, which implies that there exists a $\nu \in \R$ such that for any $\varphi \in X$,
$$
 2 \int_{\R^N} \nabla u \cdot \nabla \varphi \, dx
- a(p-2) \int_{\R^N}|u|^{p-2}u \varphi \, dx=   2\nu \int_{\R^N}u \varphi\,dx,
$$
namely that $u$ solves
$$
- \Delta u  - a \frac{(p-2)}{2} |u|^{p-2}u = \nu u.
$$
At this point from Lemma \ref{lem_Q=0} we deduce that
\begin{equation}\label{two}
A(u) = \frac{a(p-2)^2}{2p}C(u).
\end{equation}
Then on one hand, since  $Q(u) =0$ we obtain that $A(u) =k_0$. On the other hand (\ref{two}) implies that $t_u^* =1$. Thus, from the definition of $V$, one deduce that $u \notin V$ which contradicts our assumption.
\end{proof}

\begin{lem}\label{l2}
Assume that $\gamma <0$, $a >0$ and $p<4$. If $ a < K_2 \gamma^{\frac{4-p}{2}} c^{3-p}$ then it holds that $\Lambda(c) \subset V$. In particular 
$\Lambda(c) $ is a submanifold, of class $C^1$, of codimension $2$ of $X$ and a submanifold of codimension $1$ in $S(c)$.
\end{lem}

\begin{proof}
If $u \in \Lambda(c)$, then  $\phi_u(1) = Q(u) =0$. Since $t^*_u$ is the minimum point of $\phi_u$, we deduce that  
 $Q(u^{t^*_u} ) = \phi_u({t^*_u}) \leq 0$. Since $ a < K_2 \gamma^{\frac{4-p}{2}} c^{3-p}$ we deduce from Lemma \ref{ilu} and \ref{lem_A(u)=k_0_2} that $Q(u^{t^*_u} ) =  0$ is not possible. Thus $Q(u^{t^*_u} )  < 0$ and we get from Lemma \ref{vuoto} that  $u \in V$.	
\end{proof}

From now on we assume that $a < K_2\; \gamma^\frac{4-p}{2}\; c^{3-p}. $
In view of Lemma \ref{coercive2} we can define
	$$\gamma^+(c) := \sup_{\Lambda^+(c)} F(u) \quad \mbox{and} \quad \gamma^-(c) := \sup_{\Lambda^-(c)} F(u).$$
Aiming to prove Theorem \ref{infreach233} we shall establish the existence of a Palais-Smale sequence $(u_n) \subset \Lambda^{+}(c)$ (respectively $(u_n) \subset \Lambda^{-}(c)$)  for $F$ restricted to $S(c)$. Arguing as in Section 4, we define the  two functionals
$$I^+ : V \mapsto \R \quad \mbox{by} \quad I^+(u) = F(u^{s_{u}^+}) \quad \mbox{and} \quad
I^- : V \mapsto \R \quad \mbox{by} \quad I^-(u) = F(u^{s_{u}^-}).$$
By  Lemma \ref{regularite2}, the maps $u \mapsto s_{u^{+}}$  and $u \mapsto s_{u^{-}}$ are of class $C^1$ and thus  the functionals $I^+$ and $I^-$ are of class $C^1$.
As in Section 4, we can prove the following results.

\begin{lem}\label{isomorphism2}
	The maps $T_u V \rightarrow T_{ u^{s^+_{u}}} V$ defined by  $\psi \rightarrow  \psi^{s^+_{u}} $  and
	$T_u V \rightarrow T_{ u^{s^-_{u}}} V $ defined by $\psi \rightarrow \psi^{s^-_{u}} $
	are isomorphisms. 
\end{lem}

\begin{lem}\label{transform2}
	We have that $dI^+(u)[\psi]=dF(u^{s^+_u}) [\psi^{s^+_{u}}]$ for any $u \in V$, $\psi  \in T_u V$
	 and $dI^-(u)[\psi]=dF(u^{s^-_u}) [\varphi^{s^-_{u}}]$ for any $u\in V$ and $\psi \in T_u V$.
\end{lem}

In our next lemma $I^{\pm}$ denotes either $I^+$ or $I^-$ and accordingly $\Lambda^{\pm}(c)$ denotes $\Lambda^+(c)$ (or $\Lambda^-(c)$) and $s_u = s_{u}^+$ (or $s_u = s_{u}^-$). This lemma is crucial to guarantee that it is possible to develop a minimax argument inside $V$.

\begin{lem}\label{boundary}
Assume that $\gamma <0$, $p<4$ and let $$K_1\; \gamma^\frac{4-p}{2}\; c^{3-p} < a < K_2\; \gamma^\frac{4-p}{2}\; c^{3-p}. $$	
	If $(v_n) \subset V$ is a sequence with $v_n \to v_0 \in \partial V$ strongly in $X$, then $I^\pm(v_n) \to - \infty$.
\end{lem}
\begin{proof}
	Let
	$(v_n) \subset V$ such that $v_n \to v_0 \in \partial V$ strongly in $X$, as $n \to  \infty$.
	
	Since $v_n \in V$,  we have $(t^*_{v_n})^2 A(v_n) > k_0$ and
	$s_{v_n}^- \leq  t^*_{v_n} \leq s_{v_n}^+$. Moreover since $v_0 \in \partial V$, we have 
	$(t^*_{v_0})^2 A(v_0) = k_0$,  $t^*_{v_0} \neq 0$ and 
	$
	\limsup  s_{v_n}^- \leq t^*_{v_0} \leq \liminf  s_{v_n}^+$.
	
	Now if $(s_{v_n}^+)$ is bounded from above, then  up to a subsequence, it converges to
	$\bar s \neq 0$ and $t^*_{v_0} \leq \bar s$. 
	Moreover
since $v_n \to v_0$ strongly in $X$ and 
	$Q(v_n^{s_{v_n}^+} )= 0$,  
	we infer that $Q(v_0^{\bar s})=0.$
	
	At this point we deduce from Lemma \ref{l2} that $v_0^{\bar s} \in V$ and thus, by Lemma \ref{dentro22} we have $v_0 \in V$ in contradiction with the assumption that $v_0 \in \partial V$.
	
	We conclude that  $(s^+_{v_n})$ is not bounded from above and thus, up to a subsequence, $s^+_{v_n} \to + \infty$, as $n \to  \infty$. Taking into account that 
	$$
	I^+ (v_n) =G(v_n^{s_{v_n}^+})=  
	- \frac{4-p}{2(p-2)}
	(s_{v_n}^+)^2\; A(v_n) -\frac{\gamma}{4} V(v_n) + c^2 \frac{\gamma}{4}\; \log(s_{v_n}^+)
	- \gamma \frac{c^2}{4(p-2)} 
	$$
	we deduce that $I^+(v_n) \to - \infty$, as $n \to  \infty$.

On the other side, up to subsequences, 
$(s_{v_n}^-)$  converges to
$\bar s$, as $n \to  \infty$. If $\bar s \neq 0$, we can argue as before, deriving a contradiction.  It follows that 
$s_{v_n}^- \to 0^+$ as $n \to \infty$.
Taking into account that 
$$
I^- (v_n) =G(v_n^{s_{v_n}^-})=  
- \frac{4-p}{2(p-2)}
(s_{v_n}^-)^2\; A(v_n) -\frac{\gamma}{4} V(v_n) + c^2 \frac{\gamma}{4}\; \log(s_{v_n}^-)
- \gamma \frac{c^2}{4(p-2)} 
$$
we deduce that $I^-(v_n) \to - \infty$, as $n \to \infty$.
\end{proof}

\begin{lem}
	\label{ps222}
	Assume that $\gamma <0$, $p<4$ and that
	 $K_1\; \gamma^\frac{4-p}{2}\; c^{3-p} < a < K_2\; \gamma^\frac{4-p}{2}\; c^{3-p}. $
	Let $\mathcal{G}$ be a homotopy stable family of compact subsets of $V$ with closed boundary $B$ and let
$$
\textcolor{blue}{e_{\mathcal{G}}^\pm}
:= \sup_{A\in \mathcal{G}}\min_{u\in A}I^{\pm}(u).$$ 
	Suppose that $B$ is contained in a connected component of $\Lambda^{\pm}(c)$ and that $$\min\{\inf I^{\pm}(B),0\}>
	\textcolor{blue}{e_{\mathcal{G}}^\pm} > -\infty.$$ Then there exists a Palais-Smale sequence $(u_n) \subset \Lambda^{\pm}(c)$ for $F$ restricted to $V$ at level 	
$\textcolor{blue}{e_{\mathcal{G}}^\pm}$.
	\end{lem}

\begin{proof}
	Take $(D_n) \subset \mathcal{G}$ such that $\min_{u\in D_n}I^{\pm}(u) > 
	\textcolor{blue}{e_{\mathcal{G}}^\pm} -\frac1n$ and
	$$\eta :[0,1]\times V \rightarrow V,\ \eta (t,u)=u^{{1-t + ts_u}}.$$
	Since $s_u =1$ for any $u\in \Lambda^{\pm}(c)$, and $B\subset \Lambda^{\pm}(c)$, we have $\eta (t,u)=u$ for $(t,u)\in (\{0 \}\times V)\cup ([0,1]\times B)$. Observe also that $\eta$ is continuous. Then, using the definition of $\mathcal{G}$, we have
	$$A_n:= \eta (\{1 \}\times D_n )=\{u^{s_u}:\ u\in D_n \}\in \mathcal{G}.$$
	Also notice that $A_n \subset \Lambda^{\pm}(c)$ for all $n \in \N$. Let $v\in A_n$, i.e. $v=u^{s_u}$ for some $u\in D_n$ and $I^{\pm}(u)=I^{\pm}(v)$. 
	In particular we have  $\min_{A_n}I^{\pm}=\min_{D_n}I^{\pm}$ and therefore $(A_n) \subset \Lambda^{\pm}(c)$ is another maximizing sequence of $\textcolor{blue}{e_{\mathcal{G}}^\pm}$. 
	Now by  Lemma \ref{boundary}, we derive that the superlevels of $(I^\pm)^d$ are complete for any $d \in \R$. A direct adaption of the minimax principle \cite[Theorem 3.2]{Gh} implies the existence of 
	a Palais-Smale sequence $(\tilde{u}_n)$ for $I^{\pm}$ on $V$ at level 
	$\textcolor{blue}{e_{\mathcal{G}}^\pm}$
such that
	$dist_{X} (\tilde{u}_n , A_n)\rightarrow 0$ as $n\rightarrow \infty$. 
	Now  writing $s_n=s_{\tilde{u}_n}$ to shorten the notations, we set $u_n=\tilde{u}_n^{s_n}\in \Lambda^{\pm}(c)$.  We claim that there exists $C>0$ such that,
	\begin{equation}
	\label{BaSoe13}
	\frac1C \leq s_n^2\leq C
	\end{equation}
	for $n \in \N$ large enough. 
	Indeed, notice first that
	\begin{equation}
	\label{BaSoe23}
	s_n^2=\dfrac{A(u_n)}{A(\tilde{u}_n)}.
	\end{equation}
	By Gagliardo-Nirenberg inequality and
\begin{equation}\label{3222}
C(u_n) = \frac{p}{a(p-2)}\Big[ A(u_n) + \frac{\gamma c^2}{4}\Big]
\end{equation}
there exists $C>0$ such that 
	\begin{equation}
	\label{BaSoe33suoo}
	C \leq A(u_n) 
	\end{equation}
	 for $n \in \N$.
	 Moreover since $F(u_n)=I^{\pm}(\tilde{u}_n )\rightarrow 
	 \textcolor{blue}{e_{\mathcal{G}}^\pm}
	 $, we  know from Lemma \ref{coercive2} (ii), that there exists $M>0$ such that $A(u_n) \leq M$.
	Also, since $(A_n) \subset \Lambda^{\pm}(c)$, is a maximizing sequence for $
	\textcolor{blue}{e_{\mathcal{G}}^\pm}$
	, we deduce, still by Lemma \ref{coercive2} (ii), that $(A_n)$ is uniformly bounded in $H$ and thus from $dist_X(\tilde{u}_n , A_n)\rightarrow 0$ as $n\rightarrow \infty$, it implies that $\sup_n A(\tilde{u}_n) <\infty$. 
	Moreover since $A_n$ is compact for every $n \in \N$, there exists a $v_n \in A_n$ such that  $dist_X(\tilde{u}_n , A_n)=\|v_n-\tilde{u}_n \|_X$  and, using  
	$(\ref{BaSoe33suoo})$, 
	  we deduce that
	$$A( \tilde{u}_n)\geq A(v_n) - A(\tilde{u}_n - v_n)\geq K$$
	for some $K >0$ and this proves the claim.

	Next, we show that $(u_n) \subset \Lambda^{\pm}(c)$ is a Palais-Smale sequence for $F$ on $V$ at level $\textcolor{blue}{e_{\mathcal{G}}^\pm}$.
	
	Denoting by $\|.\|_\ast$ the dual norm of $(T_{u_n}S(c))^\ast$ and recalling that $V$ is open in $S(c)$, we have
	$$
	\|dF(u_n)\|_\ast=
		\sup_{\psi \in T_{u_n}V,\ \|\psi\| 
		\leq 1}|dF(u_n)[\psi]| = 
	\sup_{\psi \in T_{u_n}V,\ \|\psi\| \leq 1}|dF(u_n)[ (\psi^{-s_n})^{s_n}]|.
	$$ 
	From Lemma \ref{isomorphism} we know that  $T_{\tilde{u}_n} V \rightarrow T_{u_n} V $ defined by $\psi \rightarrow \psi^{s_n}$ is an isomorphism. 
	Also, from Lemma \ref{transform} we have that $dI^{\pm}(\tilde{u}_n)[\psi^{-s_n}]=dF(u_n)[(\psi^{-s_n})^{s_n}]$. It follows that
	\begin{equation}\label{lienclef}
	\|dF(u_n)\|_\ast = \sup_{\psi \in T_{u_n}V ,\ \|\psi\| \leq 1} |dI^{\pm}(\tilde{u}_n)[\psi^{-s_n}]|.
	\end{equation}
	At this point it is easily seen from \eqref{BaSoe1} that (increasing $C$ if necessary)
	$ \|\psi^{-s_n}\| \leq C \|\psi\| \leq C$
	and we deduce from \eqref{lienclef} that $(u_n) \subset \Lambda^{\pm}(c)$ is a Palais-Smale sequence for $F$ on $V$ at level 
	$\textcolor{blue}{e_{\mathcal{G}}^\pm}$.
\end{proof}

\begin{lem}\label{psbis}
Assume that $\gamma <0$, $p<4$ and that $$K_1\; \gamma^\frac{4-p}{2}\; c^{3-p} < a < K_2\; \gamma^\frac{4-p}{2}\; c^{3-p}. $$
	There exists a Palais-Smale sequence  $(u_n) \subset \Lambda^+(c)$ for $F$ restricted to $V$ at the level $\gamma^{+}(c)$ and a Palais-Smale sequence $(u_n) \subset \Lambda^-(c)$ for $F$ restricted to $V$ at the level $\gamma^{-}(c)$.
\end{lem}

\begin{proof}
	Let us assume that $(u_n) \subset \Lambda^+(c)$, the other case can be treated similarly. We use Lemma \ref{ps222} 
	taking the set  $\bar{\mathcal{G}}$ of all singletons belonging to $V$ and $B=\varnothing$. It is clearly a homotopy stable family of compact subsets of $V$ (without boundary).  Since
$$e_{\bar{\mathcal{G}}}^{+}:=\sup_{A\in \bar{\mathcal{G}}}\min_{u\in A}I^{+}(u)=\sup_{u\in V}I^{+}(u)=\gamma^+(c)$$	
the lemma follows directly from Lemma \ref{ps222}.
\end{proof}

Now we are ready to give

\begin{proof}[Proof of Theorem \ref{infreach233}]
	We give the proof for $u^+$, the one for $u^-$ is almost identical. Let $(u_n) \subset  \Lambda^+(c)$ be a Palais-Smale sequence for $F$ restricted to $V$ at level $\gamma^+(c)$ whose existence is insured by Lemma \ref{psbis}. By Lemma  \ref{coercive2} we know that  $(u_n)$ is bounded in $H$ and that $(V_1(u_n))$ stays bounded.  Also since the functional $F$ is translational invariant, reasoning as in the proof of Lemma \ref{X-bound} it is not restrictive to assume that $(u_n) \subset \Lambda^+(c)$ is bounded in $X$. At this point we conclude using Lemma \ref{convergence}. 
\end{proof}

{\sc Address of the authors:}\\[1em]
\begin{tabular}{ll}
Silvia Cingolani & Louis Jeanjean\\
Dipartimento di Matematica & Laboratoire de Math\'ematiques (UMR 6623)\\
Università degli Studi di  Bari Aldo Moro & Universit\'{e} Bourgogne Franche-Comt\'{e}\\
Via Orabona 4  & 16, Route de Gray\\
70125  Bari, Italy  & 25030 Besan\c{c}on Cedex, France\\
Email : silvia.cingolani@uniba.it  & Email :  louis.jeanjean@univ-fcomte.fr
\end{tabular}


\begin{thebibliography}{00}

\bibitem{AmRu}{A. Ambrosetti, D. Ruiz:}
{\em Multiple bound states for the Schr\"odinger-Poisson equation,} Commun. Contemp. Math. 10 (2008) 1-14.


\bibitem{AzPo}{A. Azzollini, A. Pomponio:}
{\em Ground state solutions for the nonlinear Sch\"rodinger-Maxwell equations,}  J. Math. Anal. Appl. 345 (2008) 90-108.


\bibitem{AmMa}{A. Ambrosetti, A. Malchiodi:}
{\em Nonlinear analysis and semilinear elliptic problems,}
Cambridge Studies in Advanced Mathematics. Cambridge University Press, Cambridge, 2007.







\bibitem{BaSo1} {T. Bartsch, N. Soave:}
{\em Correction to: ``A natural constraint approach to normalized solutions of nonlinear Schr\"odinger equations and systems",} 
[J. Funct. Anal. 272 (2017) 4998-5037], J. Funct. Anal. 275 (2018) 516-521. 
 
 

\bibitem{BaSo2} {T. Bartsch, N. Soave:}
{\em Multiple normalized solutions for a competing system of Schr\"odinger equations,} 
Calc. Var. (2019) 58: 22  https://doi.org/10.1007/s00526-018-1476-x

 


\bibitem{BatVanS} {L. Battaglia, J. Van Schaftingen:}
{\em   Groundstates of the Choquard equations with a sign-changing self-interaction potential,} Z. Angew. Math. Phys. 69 (2018) 16 pp. 


\bibitem{BJ2} {J. Bellazzini, L. Jeanjean:}
{\em   On dipolar quantum gases in the unstable regime,}
SIAM J. Math. Anal. 48  (2016) 2028-2058.

\bibitem{BeJeLu} {J. Bellazzini, L. Jeanjean, T. Luo :}
{\em  Existence and instability of standing waves with prescribed norm for a class of Schr\"odinger-Poisson equations,}
Proc. Lond. Math.. Soc. 107 (2013) 303-339.

\bibitem{BeLi1} {H. Berestycki, P.-L. Lions:}
{\em Nonlinear scalar field equations. I. Existence of a ground state,}
Arch. Rational Mech. Anal. 82(4) (1983) 313-345.

\bibitem{BeLi2} {H. Berestycki, P.-L. Lions:}
{\em Nonlinear scalar field equations. II. Existence of a ground state,}
Arch. Rational Mech. Anal. 82(4) (1983) 347-375.


\bibitem{BoCiVa} D. Bonheure, S. Cingolani, J. Van Schaftingen: {\em The logarithmic Choquard equation: sharp asymptotics and nondegeneracy of the groundstate,} J. Funct. Anal. 272 (2017)
5255-5281.


\bibitem{CiCaSe} S. Cingolani, M. Clapp, S. Secchi : {\em Multiple solutions to a magnetic nonlinear Choquard equation},
		Z. Angew. Math. Phys. 63 (2012) 233--248.
		
\bibitem{CiWe} S. Cingolani, T. Weth: {\em  On the planar Schr\"odinger-Poisson system,} Ann. Inst. H. Poincar\'e Anal. Non Lin\'eaire 33 (2016) 169-197. 



\bibitem{DuWe} {M. Du, T. Weth:} {\em Ground states and high energy solutions of the planar Schr\"odinger-Poisson system,}
Nonlinearity 30 (2017) 3492-3515.


\bibitem{Gh} {N. Ghoussoub:}
{\em Duality and perturbation methods in critical point theorey,} Cambridge University Press, (1993).


\bibitem{GoJe}{T. Gou, L. Jeanjean:}
{\em Multiple positive normalized solutions for nonlinear Schr\"odinger systems,}
Nonlinearity 31 (2018) 2319-2345.


\bibitem{harrison} {R. Harrison, T. Moroz, K.P. Tod:}
{\em A numerical study of the Schr\"odinger--Newton equation}, Nonlinearity 16 (2003) 101-122.




\bibitem{JeLu} {L. Jeanjean, T. Luo:}
{\em Sharp nonexistence results of prescribed $L^2$-norm solutions for some class of Schr\"odinger-Poisson
and quasi-linear equations}, Z. Angew. Math. Phys. 64 (2013) 937-954.

\bibitem{JeLuWa} {L. Jeanjean, T. Luo, Z-Q Wang:}
{\em Multiple normalized solutions for quasi-linear Schr\"odinger equations,}
J. Differential Equations 259(8) (2015) 3894-3928.

\bibitem{Lieb} {E. H. Lieb:}
{\em Existence and uniqueness of the minimizing solution of Choquard's nonlinear equation}, Stud. Appl. Math. 57 (1977)  93-105.

\bibitem{Lions} {P.-L. Lions:}
{\em Solutions of Hartree-Fock equations for Coulomb systems}, Comm. Math. Phys. 109 (1984) 33-97.


\bibitem{luo}{T. Luo:}
{\em Multiplicity of normalized solutions for a class of nonlinear Schr\"odinger-Poisson-Slater equations,} J. Math. Anal. Appl. 416 (2014) 195-204.


\bibitem{morozschaftingen}{V. Moroz,  and J. Van Schaftingen:} {\em Groundstates of nonlinear Choquard equations: existence, qualitative properties and decay asymptotics},
		J. Functional Anal. 265 (2013) 153-302.


\bibitem{pekar} {S.I. Pekar:}
{\em Untersuchungen \"uber die Elektronentheorie der Kristalle}, Akademie-Verlag, Berlin 1954 29-34.


\bibitem{penrose} {R. Penrose:}
{\em On gravity's role in quantum state reduction}, {Gen. Rel. Grav.} {28.}


\bibitem{So1}{N. Soave:}
{\em Normalized ground states for the NLS equation with combined nonlinearities,} arXiv:1811.00826


\bibitem{Stubbe} J. Stubbe: Bound states of two-dimensional Schr\"odinger-Newton equations. arXiv:0807.4059v1, 2008.

\bibitem{Ta}{G. Tarantello:}
{\em On nonhomogeneous elliptic equations involving critical Sobolev exponent,} Ann. Inst. H. Poincar\'e Anal. Non Lin\'eaire 9 (1992) 281-304.



\end{thebibliography}
 \end{document}